\documentclass[11pt]{amsart}

\usepackage{xypic}  
\usepackage{amssymb}
\usepackage{amsthm}
\usepackage{epsfig}
\usepackage{paralist, comment}
\usepackage{lmodern}
\usepackage{color}
\usepackage{longtable}
\usepackage{graphicx}

\usepackage{comment}

\usepackage[all,cmtip]{xy}
\usepackage{amsmath}
\usepackage{xcolor}

\newtheorem{thm}{Theorem}[section] 

\newtheorem{lemma}[thm]{Lemma}

\newtheorem{proposition}[thm]{Proposition}

\newtheorem{corollary}[thm]{Corollary}
\theoremstyle{definition}

\newtheorem{remark}[thm]{Remark}

  \newtheorem{definition-remark}[thm]{Definition-Remark}

\def\geq{\geqslant}
\def\leq{\leqslant}

\def\im{\operatorname{im}}
\def\SL{\operatorname{SL}}

\def\rank{\operatorname{rank}}

\def\c1{\operatorname{c_1}}
\def\c2{\operatorname{c_2}}

\def\Sym{\operatorname{Sym}}

\def\Sec{\operatorname{Sec}}
\def\SL{\operatorname{SL}}
\def\Tan{\operatorname{Tan}}

\def\CC{{\mathbb C}}

\def\PP{{\mathbb P}}

\def\O{{\mathcal O}}

\def\Ii{{\mathcal I}}

\def\H{{\mathcal H}}

\def\C{{\mathcal C}}

\def\c{\mathfrak{c}}

\def\cong{\simeq}

\def\+{\oplus}               
\def\*{\otimes}                  

\def\Hess{\operatorname{Hess}}
\def\hess{\operatorname{hess}}

\def\det{\operatorname{det}}
\def\Sing{\operatorname{Sing}}

\def\GN{\operatorname{GN}}
\def\Sym{\operatorname{Sym}}

\begin{document}

\title{The Hessian map}

\author[C.~Ciliberto]{Ciro Ciliberto}
\address{Ciro Ciliberto, Dipartimento di Matematica, Universit\`a di Roma Tor Vergata, Via della Ricerca Scientifica, 00173 Roma, Italy}
\email{cilibert@mat.uniroma2.it}

\author[G.~Ottaviani]{Giorgio Ottaviani}
\address{Giorgio Ottaviani, Dipartimento di Matematica e Informatica ``Ulisse Dini'', Universit\`a di Firenze,  Viale Morgagni, 67/A, 50134 Firenze, Italy}
\email{giorgio.ottaviani@unifi.it}



 \begin{abstract}  In this paper we study the Hessian map $h_{d,r}$ which associates to any hypersurface of degree $d$ in $\PP^r$ its Hessian hypersurface. We study general properties of this map and we prove that: $h_{d,1}$ is birational onto its image if $d\geq 5$; we study in detail the maps $h_{3,1}$, $h_{4,1}$ and $h_{3,2}$; we study the restriction of the Hessian map to the locus of hypersurfaces of degree $d$ with Waring rank $r+2$ in $\PP^r$, proving  
 that this restriction is injective as soon as $r\geq 2$ and $d\geq 3$, which implies that $h_{3,3}$ is birational onto its image; we prove that the differential of the Hessian map is of maximal rank on the generic hypersurfaces of degree $d$ with Waring rank $r+2$ in $\PP^r$, as soon as $r\geq 2$ and $d\geq 3$.\end{abstract}

\maketitle

\tableofcontents

\vspace{-1cm}

\section{Introduction}\label{sec:hessian}
Let $\Sigma(d,r)$ be the projective space of dimension $N(d,r):={{d+r}\choose d}-1$, which parameterizes the hypersurfaces of degree $d$ in $\PP^r$. If $\PP^r=\PP(V^\vee)$, with $V$ a $\mathbb C$ vector space of dimension $r+1$, then $\Sigma(d,r)=\PP({\rm Sym}^d(V)))$.

 If $d\geq 3$, consider the  rational map
\[
h_{d,r}: \Sigma(d,r)\dasharrow \Sigma ((r+1)(d-2),r),
\]
called the \emph{Hessian map}, which maps a hypersurface $F$ to its \emph{Hessian hypersurface} $\Hess(F)$. 

If one introduces in $\PP^r$ a system of homogeneous coordinates $[x_0,\ldots, x_r]$, and if  $F$  in { these} coordinates is defined by an equation $f=0$, where $f$ is a homogeneous polynomial of degree $d$ in $x_0,\ldots, x_r$, then $\Hess(F)$ is defined by the equation
\[
\det \Big ( \frac {\partial^2 f}{\partial x_i\partial x_j} \Big)_{0\leq i\leq j\leq r}=0.
\] 
The polynomial on the left side is called the \emph{Hessian polynomial} of $f$, and denoted by $\hess(f)$. We will often denote the derivatives with respect to the variables $x_0,\ldots, x_r$ with a { subscript}, e.g., 
\[
\hess(f)=(f_{ij})_{0\leq i\leq j\leq r}.
\]  

From the equation of the hessian polynomial, it follows that $h_{d,r}$ is defined by a linear system $\H_{d,r}$ 
of hypersurfaces of degree $r+1$ in  $\Sigma(d,r)$. 

The indeterminacy points of $h_{d,r}$, which are the base points for $\H_{d,r}$, are the hypersurfaces $F$ with equation $f=0$ such that $\hess(f)\equiv 0$. 
These are called \emph{hypersurfaces with vanishing Hessian}. The indeterminacy locus of $h_{d,r}$ has a natural scheme structure, it is denoted by ${\rm GN}_{d,r}$ and it is called the \emph{$(d,r)$--Gordan-Noether locus}. 

Among the hypersurfaces with vanishing Hessian there are the \emph{cones}, which fill up an irreducible closed subset $\C_{d,r}$ of $\Sigma(d,r)$ called the \emph{cone locus}. A hypersurface $F$ of degree $d$ with equation $f=0$ is a cone if and only if the derivatives  $f_i$ are linearly dependent for $i=0,\ldots,r$, i.e., if and only if the \emph{polar map} $f^*: V^\vee\to \Sym^{d-1}(V)$ determined by $f$ has rank smaller than $r+1$. This condition determines a scheme structure on $\C_{d,r}$ called the \emph{cone scheme structure} on $\C_{d,r}$. Another scheme structure on $\C_{d,r}$ is induced by it being contained in ${\rm GN}_{d,r}$. 

The characterization of the hypersurfaces with vanishing Hessian which are not cones is in general a highly non--trivial problem. We recall that:

\begin{thm}[Hesse's Theorem]\label{thm:hesse} If $r\leq 3$, then a hypersurface has vanishing hessian if and only if it is a cone.
\end{thm}

This is no longer true for $r\geq 4$. There is a long history concerning hypersurfaces with vanishing Hessian, too long to be recalled here, for some information see \cite {CilRusSim}. We will only recall that, due to results of Gordan--Noether \cite {GoNo}, Franchetta \cite {fra} and others (see  \cite {deRu}, \cite {GarRe}, \cite {Lo}), there is a full classification of hypersurfaces in $\PP^4$, not cones, with vanishing Hessian.  For instance, in degree $3$ these are the hypersurfaces that are in the  ${\rm PGL}(5, \mathbb C)$--orbit of the \emph{Perazzo cubic threefold} with equation $x_0x_3^2+x_1x_3x_4+x_2x_4^2=0$.

There are various question concerning the Hessian map which are worth to be considered. Here we list some of them:\\
\begin{inparaenum}
\item [(i)] determine the scheme structure of the Gordan--Noether locus. This is probably too ambitious, but, as particular cases, determine this scheme structure for the {cone locus} and for the Gordan--Noether locus in $\PP^4$, at least for hypersurfaces of low degree;\\
\item [(ii)] study the image and the fibres of the Hessian map. In particular, when $d<(r+1)(d-2)$, i.e., when $d>2+\frac 2r$, is the hessian map generically injective? In other words, is the general hypersurface of degree $d$ uniquely determined by its Hessian  when $d>2+\frac 2r$?\\ 
\item [(iii)] study the (closure of the) image of the Hessian map $h_{d,r}$, which may be called the \emph{Hessian variety} of type $(d,r)$, and denoted by $H_{d,r}$. 
\end{inparaenum}

The present paper is devoted to give some partial answers to some of these questions. Specifically, in \S  \ref {sec:d,1} we present some general considerations and prove several results about $h_{d,1}$, in particular we prove in Theorem \ref {thm:birat} that $h_{d,1}$ is birational onto its image $H_{d,1}$ for $d\geq 5$. This is the best possible result in this direction, because $h_{d,1}$ is not birational onto its image if $d\leq 4$. In \S\S \ref {sec:1,3}, \ref {sec:1,4} and \ref  {sec:2,3} we study in detail the maps $h_{3,1}$, $h_{4,1}$ and $h_{3,2}$. The results here are often classical and some of them well known in the current literature, however we put them in our general perspective. In \S   \ref {sec:r+2} we study the restriction of the Hessian map to the locus of hypersurfaces of degree $d$ with Waring rank $r+2$ in $\PP^r$. In Theorem \ref {prop:r+2} we prove that this restriction is injective as soon as $r\geq 2$ and $d\geq 3$. As a consequence of this and of the famous Sylvester Pentahedral Theorem, we prove in Theorem \ref {thm:sylv} that the map $h_{3,3}$ is birational onto its image. In \S \ref {sec:finite} we prove that the differential of the Hessian map is of maximal rank on the generic hypersurfaces of degree $d$ with Waring rank $r+2$ in $\PP^r$, as soon as $r\geq 2$ and $d\geq 3$ (see Theorem \ref {thm:der_hessian}). As a consequence, we have that $h_{d,r}$ is generically finite onto its image as soon as $r\geq 2$ and $d\geq 3$ (see Corollary \ref {cor:fin}).

We conjecture that  $h_{d,r}$ should be birational onto its image as soon as $r\geq 2$ and $d\geq 3$, {except for $h_{3,2}$}, but so far we have not been able to prove it. 

{In this paper we work over an algebraically closed field of characteristic zero. 
\medskip

\noindent {\bf Acknowledgements:} Both authors are members of GNSAGA of INdAM. The first author acknowledges the MIUR Excellence Department Project awarded to the Department of Mathematics, University of Rome Tor Vergata, CUP E83C18000100006. 
The second author acknowledges the H2020-MSCA-ITN-2018 project POEMA.}

\section{Some general remarks on $h_{d,1}$}\label {sec:d,1}

We start by focusing on 
$$h_{d,1}:  \Sigma(d,1)\cong \PP^d{ \dasharrow} \Sigma(2d-4,1)\cong \PP^{2d-4}.$$
Note that $\Sigma(d,1)$ parameterizes all effective divisors of degree $d$ on $\PP^1$.  The indeterminacy locus of $h_{d,1}$ is the cone locus $\C_{d,1}$, which coincides with the rational normal curve $\Gamma:=\Gamma_d$ of degree $d$ parameterizing all divisors of type $dx$ in $\Sigma(d,1)$, with $x\in \PP^1$. 

\begin{remark}\label{rem:con} The cone scheme structure on the rational normal curve $\Gamma$ is the reduced scheme structure. In fact, let 
\begin{equation}\label{eq:pol}
f(x_0,x_1)={ \sum_{i=0}^da_i{d\choose i}x_0^{d-i}x_1^i}
\end{equation}
be a generic homogeneous polynomial of degree $d$. It defines a cone if and only if the  derivatives
\[
{ \begin{split}
& f_0=d\sum_{i=0}^da_i{{d-1}\choose i}x_0^{d-1-i}x_1^i\\
&f_1=d\sum_{i=0}^da_{i+1}{{d-1}\choose i}x_0^{d-1-i}x_1^i
\end{split} }
\]
are linearly dependent. Hence the cone scheme structure is defined by the equations
\[
{ \rank \left(\begin{matrix} 
a_0&a_1& \cdots &a_{d-1} \\
a_1&a_2&\cdots &a_d \\
\end{matrix}\right)<2, }
\]
which define the reduced rational normal curve with affine parametric equations
\[
{ a_i=  t^i}, \quad \text {for}\quad i={ 0},\ldots, d, \quad \text {and}\quad t\in \mathbb C.
\]
\end{remark}\medskip

\begin {proposition}\label{prop:GN} The scheme structure determined by $\GN_{d,1}$ on $\Gamma$ is the reduced scheme structure, i.e., the quadrics in the linear system $\H_{d,1}$ defining $h_{d,1}$ cut out  $\Gamma$ schematically. 
\end{proposition}

\begin{proof} We need to prove that at any point of $\Gamma$ the tangent hyperplanes to the  quadrics in the linear system $\H_{d,1}$ intersect only along the tangent line to $\Gamma$ at that point.  Since $\SL({ 2},\CC)$ acts transitively on $\Gamma$, it suffices to prove the assertion at a specific point of $\Gamma$, e.g., at the point $p$ of $\Gamma$ corresponding to the polynomial $f$ as in \eqref {eq:pol} with $a_0=1$ and  $a_i=0$ for $i>0$, i.e., the polynomial $f=x_0^d$. 

The linear polynomials in the variables $a_0,\ldots, a_d$ defining the tangent spaces in question are the coefficients in the variables $x_0,x_1$ of the polynomial  
$\sum_{i=0}^da_i\frac{\partial \mathrm{hess}(f)}{\partial a_i}(p)$. We have
\[
\begin{split}
&\sum_{i=0}^da_i\frac{\partial \mathrm{hess}(f)}{\partial a_i}(p)=\\
&=\sum_{i=0}^da_i\frac{\partial f_{00}}{\partial a_i}(p)f_{11}(p)+
f_{00}(p)\sum_{i=0}^da_i\frac{\partial f_{11}}{\partial a_i}(p)-2\sum_{i=0}^da_i\frac{\partial f_{01}}{\partial a_i}(p)f_{01}(p)=\\
&=d(d-1)x_0^{d-2}\sum_{i=2}^da_ii(i-1)x_0^{d-i}x_1^{i-2}
\end{split}
\]
because $f_{11}(p)=f_{01}(p)=0$. From the above relations, we see that the vanishing of the equations defining the tangent hyperplanes to the  quadrics in the linear system $\H_{d,1}$ give the solution $a_i=0$ for $i=2, \ldots, d$, which are just the equations of the tangent line to $\Gamma$ at $p$. \end{proof}

Since the linear system $\H_{d,1}$ defining $h_{d,1}$ consists of quadrics containing $\Gamma$, then all chords (and tangents) of $\Gamma$ are contracted to points by $h_{d,1}$. More precisely we have:

\begin{proposition}\label{prop:veron} The image via $h_{d,1}$ of $\Sec(\Gamma)-\Gamma$ is the $(d-2)$--Veronese image of $\PP^2$. 
\end{proposition} 

\begin{proof} Let $x=[u_0,v_0]$, $y=[u_1,v_1]$ be two distinct points of $\PP^1$. The corresponding points on $\Gamma$ are the divisors $dx$ and $dy$, which have equations
\[
\alpha^d=0 \quad \text {and}\quad \beta^d=0, \quad \text{where}\quad 
\alpha =v_0x_0-u_0x_1 \quad \text{and}\quad \beta=v_1x_0-u_1x_1
\]
The chord of $\Gamma$ joining $dx$ and $dy$, is the pencil of divisors with equations
\[
\lambda \alpha^d+\mu \beta^d=0, \quad \text{with}\quad [\lambda,\mu]\in \PP^1.
\]
A straightforward computation shows that the image of all  divisors in this chord is the divisor $(d-2)(x+y)$, with equation $(\alpha\beta)^{d-2}$.

We have the isomorphism $\phi: \Sigma(2,1)\to \PP^2$, which sends the  divisor of degree 2 of $\PP^1$ defined by an equation of the form $ax_0^2+bx_0x_1+cx_1^2=0$ to the point of $\PP^2$ with homogeneous coordinates $[a,b,c]$. Let us interpret $\PP^2$ as 
$\Sigma(1,2)$, i.e., as a dual plane. So  $\phi$ can be interpreted as the map sending the degree 2 divisor with equation  $ax_0^2+bx_0x_1+cx_1^2=0$ to the line $ax_0+bx_1+cx_2=0$. 

We have also an obvious morphism $\gamma: \Sec(\Gamma)\to \Sigma(2,1)$, which send all the points on the chord joining two points $x,y$ of $\Gamma$ to the degree 2 divisor $x+y$ of $\PP^1\cong \Gamma$. By the above considerations, the restriction of $h_{d,1}$ to $\Sec(\Gamma)-\Gamma$, can be interpreted as the composition of the morphism $\gamma$, restricted to $\Sec(\Gamma)-\Gamma$, followed by $\phi$, followed by the $(d-2)$--Veronese map of $\PP^2$. The assertion follows. \end{proof}

\begin{corollary}\label{cor:tan} The image via $h_{d,1}$ of $\Tan(\Gamma)-\Gamma$ is a rational normal curve of degree $2(d-2)$.
\end{corollary}

\begin{proof} Following the argument in the proof of Proposition \ref {prop:veron}, we see that the image of $\Tan(\Gamma)-\Gamma$ coincides with the image under the $(d-2)$--Veronese map of $\PP^2$ of the 2--Veronese image of $\PP^1$ in $\PP^2$. The assertion follows. \end{proof}

By Proposition \ref {prop:GN}, the map $h_{d,1}$ factors through the blow--up $\widetilde \Sigma(d,1)$ of $\Sigma(d,1)\cong \PP^d$ along $\Gamma$ and a morphism 
$\tilde h_{d,1}: \widetilde \Sigma(d,1)\to  \Sigma(2d-4,1)$. 
Let $E$ be the exceptional divisor of this blow--up. For the normal bundle $N_{\Gamma|\PP^d}$ we have
\[
N_{\Gamma|\PP^d}\cong \O_{\PP^1}(d+2)^{\oplus d-1},
\]
hence
\[
E=\PP(N_{\Gamma|\PP^d}^\vee)= \PP(\O_{\PP^1}(-d-2)^{\oplus d-1})\cong \PP(\O_{\PP^1}^{\oplus d-1})=\PP^1\times \PP^{d-2}.
\]

\begin{proposition}\label{prop:exc} 
For each point in $\Gamma$, corresponding to a divisor $dx$, with $x\in \PP^1$, the image via $\tilde h_{d,1}$ of the $\PP^{d-2}$ fibre of $E$ over $dx\in \Gamma$, consists of all divisors of the form $(d-2)x+D$, with $D$ any divisor in $\Sigma(d-2,1)$.
\end{proposition}

\begin{proof} To understand the image of $E$ via the Hessian map, we consider a point of $\Gamma$, corresponding to the divisor of equation
\[
\alpha^d=0, \quad \text {with}\quad \alpha=u_0x_0+u_1x_1,
\]
i.e., the point corresponding to the divisor $dx$, with $x=[u_1,-u_0]$. Take $f(x_0,x_1)$ a general homogeneous polynomial of degree $d$, and consider the pencil of divisors with equations
\[
\alpha^d+tf=0, \quad \text{with}\quad t\in \CC.
\]
A straightforward calculation shows that the limit of the Hessian of the polynomial in this pencil when $t$ tends to $0$, is the divisor with equation
\[
\alpha^{d-2}\Big (u_0^2\frac {\partial^2f}{\partial x_1^2}+u_1^2\frac {\partial^2f}{\partial x_0^2}-2u_0u_1\frac {\partial^2f}{\partial x_0\partial x_1}\Big)=0.
\]
The polynomial in parenthesis is the second polar of $f$ with respect to the point $x=[u_1,-u_0]$, and since $f$ is a general polynomial of degree $d$, it is a general polynomial of degree $d-2$. The assertion follows.  \end{proof}

\begin{proposition}\label{thm:fin} The Hessian map $h_{d,1}$ is generically finite onto its image, unless $d=3$, in which case it has general fibres of dimension 1, i.e., the chords of $\Gamma$.
\end{proposition}

\begin{proof} In the case $d=3$, the domain of $h_{3,1}$ is $\Sigma(3,1)\cong \PP^3$ and the range is $\Sigma(2,1)\cong \PP^2$. So the map cannot be generically finite.  Actually by Proposition \ref {prop:exc}, the image of $E$ is all of $\Sigma(2,1)$ and it has general fibres of dimension 1. Since $\Sec(\Gamma)=\PP^3$, also Proposition \ref {prop:veron} tells us that the image is all of $\Sigma(2,1)$. The fibres are the chords of $\Gamma$. 

Assume now $d>3$. 
In order to prove the proposition it suffices to exhibit a homogeneous polynomial of degree $d$ in $x_0,x_1$, such that the differential of $h_{d,1}$ there is of maximal rank $d$. 

Let  { $f$ be a general polynomial of degree $d$ like in \eqref {eq:pol}}. Then, up to a factor, $\hess(f)$ is
$$\det \begin{pmatrix}\sum_{i=0}^{d-2}{{d-2}\choose i}a_ix_0^{d-2-i}x_1^{i}&\sum_{i=0}^{d-2}{{d-2}\choose i}a_{i+1}x_0^{d-2-i}x_1^{i}\\
\sum_{i=0}^{d-2}{{d-2}\choose i}a_{i+1}x_0^{d-2-i}x_1^{i}&\sum_{i=0}^{d-2}{{d-2}\choose i}a_{i+2}x_0^{d-2-i}x_1^{i}\end{pmatrix}.$$

Computing the determinant we get that the coefficient of $x_0^{2d-4-p}x_1^{p}$ is
$$Q_p=\sum_{i=0}^p{{d-2}\choose i}{{d-2}\choose{p-i}}a_ia_{p-i+2}-\sum_{i=0}^p{{d-2}\choose i}{{d-2}\choose{p-i}}a_{i+1}a_{p-i+1}.$$
Computing the derivative with respect to $a_q$ we get the enter of place $(p,q)$ of the Jacobian matrix of $h_{d,1}$ which is
$$\left[{{d-2}\choose q}{{d-2}\choose{p-q}}+{{d-2}\choose {q-2}}{{d-2}\choose{p-q+2}}-2{{d-2}\choose {q-1}}{{d-2}\choose{p-q+1}}\right]a_{p-q+2}$$
We claim that when $p=q$ the coefficients in square brackets are all nonzero for $d\geq 4$ and  $d\neq 8$. This is equivalent to say
that the rank of the Jacobian at $x_0^{d-2}x_1^{2}$ is maximal, as wanted.  

First of all, one verifies directly that the assertion holds for $p=q=0$ and $p=q=d$. For  $0< p=q<d$ the coefficients are
\[
\begin{split}
&\left[{{d-2}\choose q}+{{d-2}\choose {q-2}}{{d-2}\choose{2}}-2{{d-2}\choose {q-1}}{{d-2}\choose{1}}\right]=\\
&={{d-2}\choose {q-1}}\left[\frac{(d-1)(dq^2-5dq+2d+4q)}{2q(d-q)}\right].
\end{split}
\]
Thus we are reduced to show that $(dq^2-5dq+2d+4q)=dq(q-5)+2d+4q$ is nonzero.
This is immediate when $q\ge 5$. It remains to check the cases:\\
\begin{inparaenum}
\item[$\bullet$]$q=1$, here we get $-2d+4$ which is nonzero when $d\geq 4$;\\
\item[$\bullet$]$q=2$, here we get $-4d+8$ which is nonzero when $d\geq 4$;\\
\item[$\bullet$]$q=3$, here we get $-4d+12$ which is nonzero when $d\geq 4$;\\
\item[$\bullet$]$q=4$, here we get $-2d+16$ which is nonzero when $d\geq 4$ except for  $d=8$.
\end{inparaenum}

In the case $d=8$, with similar arguments one sees that the differential of $h_{8,1}$ has maximal rank at $x_0^3x_1^5$. We leave the details to the reader. \end{proof}

We can be even more precise, and prove the following:

{
\begin{proposition}\label{prop:fiber} One has:\\
\begin{inparaenum}
\item[$\bullet$] the fiber $h_{d,1}^{-1}(h_{d,1}(x_0^2x_1^{d-2}))$ consists schematically of the single point $\{x_0^2x_1^{d-2}\}$ for $d\ge 5$, $d\neq 8$;\\
\item[$\bullet$] the fiber $h_{8,1}^{-1}(h_{8,1}(x_0^3x_1^{5}))$ consists schematically of the single point $\{x_0^3x_1^{5}\}$ .
\end{inparaenum}
\end{proposition} }
\begin{proof} We prove the assertion for $d\neq 8$. The case $d=8$ can be worked similarly and we leave the details to the reader, { we just note here that {$h_{8,1}^{-1}(h_{8,1}(x_0^2x_1^{6}))$} consists of the point $\{x_0^2x_1^{6}\}$
with multiplicity $4$.}
As in the proof of Proposition \ref {thm:fin}, we see that the differential of $h_{d,1}$ is of maximal rank at $f=x_0^2x_1^{d-2}$. 
The goal is to show that the fiber of $h_{d,1}$ at $f=x_0^2x_1^{d-2}$ consists of $f$ alone.
This form has $a_i=0$ for $i\neq d-2$ and  $\hess(f)=x_0^2x_1^{2d-6}$ up to a scalar.
Recall the notation $Q_p$ for the coefficients of $x_0^{2d-4-p}x_1^{p}$ in the Hessian (see the proof of Proposition \ref {thm:fin}). Then the fiber of $f=x_0^2x_1^{d-2}$ is cut out by the equations $Q_p=0$ for all $p=0,\ldots, 2d-4$ except $p=2d-6$. 

We consider a form { $f$ as in \eqref {eq:pol}} in such a fiber.
Suppose first that $a_0\neq 0$, so that we may assume $a_0=1$.
Then we show that there exists a $t\in \CC$ such that $a_i=t^i$, namely
$f=(x+ty)^d$, corresponding to a point on the rational normal curve $\Gamma$, which is not possible.
Indeed set $a_1=t$.
The equation $Q_0=0$ is $a_0a_2-a_1^2=0$, which implies $a_2=t^2$. However it is important for us to record that
there is a unique possible $a_2$ given $a_0=1$ and $a_1$.
The equation $Q_1=0$ is $(d-2)a_0a_3+\cdots=0$ where the other monomials involve $a_i$ with $i\leq 2$. Hence there is a unique $a_3$ given $a_0=0$ and $a_1, a_2$.
Continuing in this way, at step $p$ the equation $Q_p=0$ is
${{d-2}\choose p}a_0a_{p+2}+\cdots=0$ where the other monomials involve $a_i$ with $i\leq p+1$. Hence there is a unique $a_{p+2}$ given $a_0=1$ and $a_1,\ldots, a_{p+1}$.

This argument may be continued until we get to the equation $Q_{d-2}=0$, since $d-2<2d-6$ for $d\geq 5$.
We conclude that given $a_0=1$ the $a_i$ are uniquely determined, and this forces
$a_i=t^i$ since this is a solution.

The second step is to show that if $a_0=0$ then $a_i=0$ for $i\neq d-2$, which will prove the theorem.
The equation $Q_0=0$ is $a_0a_2-a_1^2$ and forces now $a_1=0$.
The equation $Q_2=0$ has the form ${ \left[{{d-2}\choose 2}-(d-2)^2\right]}a_2^2+\cdots=0$ where the other monomials
contain $a_0$ or $a_1$. Hence we get $a_2=0$.
The equation $Q_4=0$ has the form ${ \left[{{d-2}\choose {3}}{{d-2}\choose {1}} -{{d-2}\choose 2}^2\right]}a_3^2+\cdots=0$ where the other monomials
contain at least one among $a_0,\ldots,  a_2$. Hence we get $a_3=0$.
Continuing in this way we consider for $k=0,\ldots d-4$ the equation 
$Q_{2k}=0$ which has the form ${ \left[{{d-2}\choose {k+1}}{{d-2}\choose {k-1}} -{{d-2}\choose k}^2\right]}a_{k+1}^2+\cdots=0$ where the other monomials
contain at least one among $a_0,\ldots a_k$. Hence we get $a_{k+1}=0$.

Summing up, we get $a_i=0$ for $i\leq d-3$.
Recall we cannot use the equation $Q_{2d-6}=0$. We concentrate now on the last two equations
$$Q_{2d-5}= -a_{d-2}a_{d-1}+a_{d-3}a_d=0, \quad 
Q_{2d-4}= -a_{d-1}^2+a_{d-2}a_d=0.$$

From $Q_{2d-5}=0$ we get $a_{d-2}a_{d-1}=0$, then one of the two factors vanishes.
If $a_{d-2}=0$ then from $Q_{2d-4}=0$ we get $a_{d-1}=0$, hence $f=x_1^d$ that correspond to a point on the rational normal curve, which may be excluded as before. If $a_{d-2}\neq 0$ and $a_{d-1}=0$
then from  $Q_{2d-4}=0$ we get $a_d=0$, so that $f=x_0^2x_1^{d-2}$, as we wanted.
\end{proof}

{ The previous Proposition is the basis for the next  fundamental
Theorem. 
\begin{thm}\label{thm:birat}
The Hessian map $h_{d,1}$ is birational onto its image for $d\ge 5$.
\end{thm}
\begin{proof}
We first show that the degree of $h_{d,1}$ is one or two.

 { We recall the resolution of indeterminacies
$\tilde h_{d,1}: \widetilde \Sigma(d,1)\to  \Sigma(2d-4,1)$ described before Proposition \ref{prop:exc},
where $E=\PP(N_{\Gamma|\PP^d}^\vee) \subset\widetilde \Sigma(d,1)$ is the exceptional divisor of the blow--up of $\Sigma (d,1)$ along $\Gamma$. 
Rephrasing Proposition \ref{prop:exc}, the fiber $E_p$ of $E$ over the point $p\in \Gamma$ corresponding to $x_1^d$, maps isomorphically via $\tilde h_{d,1}$ to the linear space consisting of all divisors with equation $x_1^{d-2}g(x_0, x_1)=0$
where $g$ is any binary form of degree $d-2$. 
As we saw, if $f$ is a general binary form of degree $d$, the map  $\tilde h_{d,1}$ takes the intersection with $E$ of the proper transform of the pencil $\langle x_1^d, f\rangle$ on  $ \widetilde \Sigma(d,1)$ to $\lim_{t\to 0}h_{d,1}(x_1^{d}+tf)=x_1^{d-2}f_{00}$. Note that $f_{00}\equiv 0$ if and only if $f$ belongs to the pencil  $\langle x_1^d, x_1^{d-1}x_0\rangle$, which can be identified with the tangent line $T_{\Gamma, p}$. Indeed, the fiber of ${N_{\Gamma|\PP^d}}$ at $p$ is the quotient of the space of lines through $p$
modulo $T_{\Gamma, p}$.
It follows that, scheme theoretically, there is a unique  point  $v\in E_{p}$ which maps via $\tilde h_{d,1}$ to 
$h_{d,1}(x_0^2x_1^{d-2})=x_0^2x_1^{2d-6}$, and one can easily check that $v$
is the intersection of $E$ with the proper transform on $\widetilde \Sigma (d,1)$ of the pencil $\langle x_1^{d},x_0^{4}x_1^{d-4}\rangle$, i.e., $x_0^2x_1^{2d-6}=
\lim_{t\to 0}h_{d,1}(x_1^{d}+tx_0^{4}x_1^{d-4})=\tilde h_{d,1}(v)$.
It follows that the fiber $\tilde h_{d,1}^{-1}(\tilde h_{d,1}(w))$, with $w=x_0^2x_1^{d-2}$, consists schematically
of two points, i.e., the points $v$ and $w$.
This implies that, as claimed, the degree of $h_{d,1}$ is one or two. 

Next we prove that $h_{d,1}$ is birational onto its image. Assume by contradiction that its degree is two.
Hence for a generic binary form $f$ of degree $d$ there is a unique binary form $f'$ such that $h_{d,1}(f)=h_{d,1}(f')$.
 Consider the (quadratic)  map $h_{d,1}$ restricted to the pencil $\langle f, f'\rangle$.
Its image is not a conic, because the two points $f$ and $f'$ are glued by $h_{d,1}$, so  $h_{d,1}$ maps the line $\langle f, f'\rangle$ to a line with degree 2.
Now we degenerate $f$ to $x_0^2x_1^{d-2}$. Correspondingly the form $f'$ degenerates to $x_1^d$ and
the pencil $\langle f, f'\rangle$ degenerates to the pencil $\langle x_1^d, x_0^2x_1^{d-2}\rangle$.
This is a contradiction because, as it is easily seen, $h_{d,1}$ is bijective and not of degree 2 when restricted to this pencil.}
\end{proof}
}

\begin{corollary}\label{cor:deg} The Hessian variety $H_{d,1}$ for $d\geq 5$, is rational of  degree $2^d-d(d-1)-2$ and dimension $d$ in $\PP^{2d-4}$.
For $d=5$ it is a hypersurface of degree $10$ in $\PP^6$.
\end{corollary}
\begin{proof} Rationality and dimension follow from Theorem \ref {thm:birat}. 
The degree is computed cutting $H_{d,1}$ with $d$ general hyperplanes $\PP^{2d-4}$.
They correspond to $d$ general quadrics in $\PP^d$ cutting schematically the rational normal curve of degree $d$ plus finitely many points, whose number is  the degree of $H_{d,1}$.
The equivalence of the rational normal curve in the intersection of $d$ quadrics is $d(d-1)+2$ by \cite[Example 9.1.1]{Fulton}. This concludes the proof.
\end{proof}

Another interesting information is given by the following:

\begin{proposition}\label{prop:smooth} If $F\in \Sigma(d,1)$ is general, then  $\Hess(F)\in \Sigma(2d-4,1)$ is a reduced divisor.
\end{proposition}

\begin{proof}  Let $f(x_0,x_1)=0$ be the equation of $F$. Assume that $\Hess(F)$ is not reduced. This means that $\hess(f)$ has some multiple root. We will prove this is not the case.
To start with, we may assume that $\hess(f)$ does not have the root $x_0=0$: indeed the Hessian is a projective covariant, hence, by acting with projective transformations we may  assume $\hess(f)$ does not vanish on any given point of $\PP^1$.

Next we compute again the differential of the Hessian map in a slightly different way than before. Let $g$ be any { homogeneous polynomial} of degree $d$ in the variables $x_0,x_1$. Consider the polynomial $f+\varepsilon g$, with $\epsilon^2=0$, which can be interpreted as a tangent vector  to $\Sigma(d,1)$ at $F$. Then
\[
\hess(f+\varepsilon g)=\hess(f)+ \varepsilon \hess(f,g)
\]  
where
\[
 \hess(f,g):= \det \left(\begin{matrix} 
g_{00} & g_{01} \\
f_{10} &  f_{11}\\
\end{matrix}\right) +
\det \left(\begin{matrix} 
 f_{00} &  f_{01} \\
 g_{10} & g_{11}\\
\end{matrix}\right)
\]
will be called the \emph{simultaneous Hessian} of $f$ and $g$, and can be interpreted as the differential of $h_{d,1}$ at $F$. If $F$ is non--reduced, then $\hess(f)$ and $\hess(f,g)$ must have a common root, whatever $g$ is. We prove this is not the case. Indeed, we can take $g=x_0^d$. Then
\[
\hess(f,g)=d(d-1)x_0^{d-2}f_{11}
\] 
If this has a common root with $\hess(f)$, since, as we assumed, $\hess(f)$ does not have the root $x_0=0$, then there is a common root of $\hess(f)$ and of $f_{11}$. 
But then this is also a root of the system $f_{01}=f_{11}=0$
which means that the polynomial $f_1$ has some double root. Since $f$ is general, this is a contradiction. \end{proof}

Finally we are interested in identifying  the linear system of quadrics $\H_{d,1}$. Since $\PP^1=\PP(V^\vee)$, then  $V$ is the vector space of linear forms in $x_0,x_1$, on which we have the obvious $\SL(2,\CC)$ action. The linear system $\H_{d,1}$, of dimension $2d-4$, corresponds to a subvector space of dimension at most $2d-3$ of $\Sym^2(\Sym^d(V))$ which is invariant by the $\SL(2,\CC)$ action. It is well known that the representation $\Sym^2(\Sym^d(V))$ of $\SL(2,\CC)$ splits as a sum of irreducible invariant subspaces as follows
\begin{equation*}\label{eq:split}
\Sym^2(\Sym^d(V))=\bigoplus_{i\geq 0}\Sym^{2d-4i}(V)
\end{equation*}
(see \cite [Ex. 11.31]{FuHar}). We can interpret $\PP(\Sym^2(\Sym^d(V)))$ as the linear system of all quadrics in the projective space $\PP(\Sym^d(V))\cong \Sigma(d,1)\cong \PP^d$. In this projective space we have the cone locus $\Gamma=\C_{d,1}$. Then 
\[
S_1:=\bigoplus_{i\geq 1}\Sym^{2d-4i}(V)
\]
identifies as the vector space of quadrics vanishing on $\Gamma$ and 
\[
S_2:=\bigoplus_{i\geq 2}\Sym^{2d-4i}(V)
\]
is the vector space of quadrics vanishing on $\Tan(\Gamma)$ (see  \cite [Ex. 11.32]{FuHar}). Then we have
\begin{equation}\label{eq:tan}
\dim (S_1/S_2)=\dim (\Sym^{2d-4}(V))=2d-3 
\end{equation}
and more precisely we have an $\SL(2,\CC)$ invariant subspace $W$ of quadrics of dimension $2d-3$ such that $S_1=S_2\oplus W$. The linear systems associated to $W$ is exactly $\H_{d,1}$.  

\begin{remark}\label{rem:tan1} We can give a geometric interpretation of \eqref {eq:tan}.

Let $\Sigma\to \Tan(\Gamma)$ be the minimal desingularization of $\Tan (\Gamma)$, the tangential surface of $\Gamma$. Then $\Sigma$ is a scroll over $\PP^1$. By abuse of notation we still denote by $\Gamma$ its strict transform on $\Sigma$, we denote by $R$ a ruling of $\Sigma$ and by $H$ the pull--back of a hyperplane section of $\Tan(\Gamma)$. By Riemann--Hurwitz formula, we have $\deg(\Tan(\Gamma))=2d-2$. On the other hand $\Gamma$ is a unisecant of the ruling on $\Sigma$, hence $H\sim \Gamma+(d-2)R$. The quadrics through $\Gamma$ cut out on $\Tan(\Gamma)$ divisors which are in the linear system
\[
|2H-2\Gamma|=|2(d-2){ R}|
\] 
because $\Gamma$ is singular for $\Tan(\Gamma)$. So we have a linear map
\[
r: S_1\to H^0(\Sigma, \O_\Sigma(2(d-2){ R}))\cong \CC^{2d-3}
\]
whose kernel is $S_2$. By Corollary \ref {cor:tan} the map $r$ is surjective, which explains 
 \eqref {eq:tan}.  \end{remark}

\section{The case $r=1, d=3$}\label {sec:1,3}

This is an easy case. We have $h_{3,1}:  \Sigma(3,1)\cong \PP^3\to \Sigma(2,1)\cong \PP^2$, and the map is surjective. The chords of the cone locus $\C_{3,1}=\Gamma$ are contracted to points. The tangents to $\Gamma$ are contracted to the points of the \emph{diagonal} conic $\Delta$ of $ \Sigma(2,1)$, with $\Delta=\{2x: x\in \PP^1\}$. The exceptional divisor $E$ of the blow--up of $\PP^3$ along $\Gamma$ is a scroll over $\Gamma$, whose rulings go to the tangents lines to $\Delta$. The 2--dimensional linear system $\H_{3,1}$ is the complete linear system of quadrics through $\Gamma$.

\section{The case $r=1, d=4$}\label {sec:1,4}

In this section we examine in details the behaviour of the map $h_{4,1}: 
\Sigma(4,1)\cong \PP^4\to \Sigma(4,1)\cong \PP^4$. 

Recall that $V$ is the vector space of linear forms in $x_0,x_1$, on which we have the natural $\SL(2,\CC)$ action. Recall also formula \eqref {eq:split} from \S \ref {sec:d,1}. In the present case we have
\[
\Sym^2(\Sym^4(V))=\Sym^{8}(V)\oplus \Sym^{4}(V)\oplus \Sym^{0}(V)
\]
where $\Sym^{0}(V)=\CC$ is the space we called $S_1$ in \S \ref {sec:d,1}, which is generated by the (unique up to a constant) non--zero quadratic polynomial vanishing on $\Tan(\Gamma)$, where $\Gamma$ is the cone locus $\C_{4,1}$, a rational normal quartic curve. Again in the notation  of \S \ref {sec:d,1}, we have $S_1= \Sym^{4}(V)\oplus \Sym^{0}(V)\cong \CC^6$, which identifies with the vector space $H^0(\Ii_{\Gamma|\PP^4}(2))$ of quadratic polynomials vanishing on $\Gamma$. The 4--dimensional linear system of quadrics $\H_{4,1}$ defining $h_{4,1}$ is associated to the 5--dimensional $\SL(2,\CC)$--invariant vector space $W\cong \Sym^{4}(V)$, which we will soon identify. 

To be very explicit, let 
\begin{equation}\label{eq:f}
f(x_0,x_1)=a_0x_0^4+4a_1x_0^3x_1+6a_2x_0^2x_1^2+4a_3x_0x_1^3+a_4x_1^4
\end{equation}
be a general homogeneous polynomial of degree 4 in $(x_0,x_1)$, which gives us a class $[f]\in \Sigma(4,1)\cong \PP^4$, and we may attribute to $[f]$ the homogeneous coordinates $[a_0,\ldots, a_4]$.

The cone locus $\Gamma=\C_{4,1}$ is described by the classes of the non--zero polynomials of the form
\[
(\alpha_0x_0+\alpha_1x_1)^4
\]
with the corresponding homogeneous coordinates
\[
[\alpha_0^4,\alpha_0^3\alpha_1,\alpha_0^2\alpha_1^2, \alpha_0\alpha_1^3, \alpha_1^4]
\]
so that the ideal of the polynomials vanishing on $\Gamma$ is generated by the minors of  order 2 of the matrix 
\begin{equation*}\label{eq:bu}
A= \left(\begin{matrix} 
a_0&a_1&a_2&a_3 \\
a_1&a_2&a_3&a_4 \\
\end{matrix}\right).
\end{equation*}
We denote by $Q_{ij}$ the minor of $A$ determined by the columns of order $i$ and $j$, with $1\leq i<j\leq 4$. Let also $[x_{ij}]_{1\leq i<j\leq 4}$ be the (lexicographically ordered) homogeneous coordinates in $\PP^5$. The 5--dimensional linear system $|H^0(\Ii_{\Gamma|\PP^4}(2))|$  determines a rational map $\mu: \PP^4\dasharrow \PP^5$, which can be written as
\[
x_{ij}=Q_{ij}\quad \text{for}\quad 1\leq i<j\leq 4.
\]
It is well know that the image of $\mu$ is a smooth quadric in $\PP^5$ (see \cite [Chapt. X, \S 3.2]{SemRo}), precisely the quadric $Q$ with equation
\begin{equation}\label{eq:q}
x_{12}x_{34}-x_{13}x_{24}+x_{14}x_{23}=0.
\end{equation}

Let us now write down the Hessian map $h_{4,1}$. The Hessian of a polynomial as in \eqref {eq:f} is
\[
\begin{split}
{ \hess(f)}&= (a_0a_2-a_1^2)x_0^4+2(a_0a_3-a_1a_2)x_0^3x_1+(a_0a_4+2a_1a_3-3a_2^2)x_0^2x_1^2+\\
&+2(a_1a_4-a_2a_3)x_0x_1^3+(a_2a_4-a_3^2)x_1^4.
\end{split}
\]
Note that
\[
\begin{split}
& a_0a_2-a_1^2=Q_{12},\,\,\, a_0a_3-a_1a_2=Q_{13},\\
& a_0a_4+2a_1a_3-3a_2^2=Q_{14}+3Q_{23},\\
&a_1a_4-a_2a_3=Q_{24}, \,\,\,a_2a_4-a_3^2=Q_{34}
\end{split}
\]
so that
\[
{ \hess(f)}=Q_{12}x_0^4+2Q_{13}x_0^3x_1+(Q_{14}+3Q_{23})x_0^2x_1^2+2Q_{24}x_0x_1^3+Q_{34}x_1^4.
\]
Hence the Hessian map is so defined
\[
h_{4,1}: [a_0,\ldots, a_4]\in \PP^4\to [Q_{12},\frac 12 Q_{13},\frac 16 (Q_{14}+3Q_{23}),\frac 12 Q_{24},Q_{34}]\in \PP^4.
\]
Consider the polynomial
\begin{equation}\label{eq:bu2}
j=\det \left(\begin{matrix} 
a_0&a_1&a_2 \\
a_1&a_2&a_3 \\
a_2&a_3&a_4
\end{matrix}\right).
\end{equation}
One has
\[
\frac {\partial j}{\partial a_0}=Q_{34}, \frac {\partial j}{\partial a_1}={ -}2Q_{24},  \frac {\partial j}{\partial a_2}=Q_{14}+3Q_{23}, \frac {\partial j}{\partial a_3}={ -}2Q_{13}, \frac {\partial j}{\partial a_4}=Q_{12}
\]
hence the Hessian map $h_{4,1}$ is nothing but the \emph{polar map} of the polynomial $j$ and $\H_{4,1}$ is the linear system of polar quadrics of the hypersurface $j=0$. Moreover, since all derivatives of $j$ vanish on $\Gamma$, then $\Gamma$ is in the singular locus of the hypersurface $j=0$. This implies that this hypersurface, of degree 3,  coincides with $\Sec(\Gamma)$. 

Consider now the projective transformation $\omega$ of $\PP^5$ with equations
\[
\begin{split}
&y_{12}=x_{34},\quad y_{34}=x_{12},\quad y_{13}={ -}2x_{24},\quad y_{24}={ -}2x_{13},\\
&\qquad y_{14}=x_{14}+3x_{23}, \quad y_{23}=x_{14}-3x_{23}.
\end{split}
\]
The Hessian map $h_{4,1}$ is nothing but the composition of the map $\mu: \PP^4\dasharrow \PP^5$ with $\omega$ with the projection from the point $p=[0,0,0,1,0,0]$ on the hyperplane $y_{23}=0$. Notice that the image {of  the quadric $Q$ with equation \eqref {eq:q} via $\omega\circ \mu$ is the quadric $Q'$ with equation}
\[
12y_{12}y_{34}-3y_{13}y_{24}+y_{14}^2-y_{23}^2=0
\]
{ which} does not pass through the point $p$. Hence we conclude with the:

\begin{thm}\label{thm:deg2} The Hessian map $h_{4,1}$ has degree 2.
\end{thm}

A form
\[
g(x_0,x_1)=b_0x_0^4+4b_1x_0^3x_1+6b_2x_0^2x_1^2+4b_3x_0x_1^3+b_4x_1^4
\]
sits in the branch locus of $h_{4,1}$ if and only if 
\[
b_0b_4-4b_1b_3+3b_2^2=0.
\]
This is { the   $\SL(2, \CC)$--invariant quadric, that we call $\mathfrak Q$}. Hence it must vanish on $\Tan(\Gamma)$, and in fact this is the case, as one  checks by a direct computation.

The ramification locus of $h_{4,1}$ is contained in the Hessian hypersurface of $j$. This is defined by a polynomial of degree 5, which does not vanish identically. The hypersurface $\Sec(\Gamma)$, with equation $j = 0$ has, by Terracini's Lemma, \emph{parabolic points}, namely, it is described by a family of lines along which the tangent hyperplane is fixed. This implies that $j$ divides $\hess(j)$ (see \cite{Cil}). The quotient is a polynomial of degree 2, which defines a quadric coinciding with the ramification locus of $h_{4,1}$. This quadric is $\SL(2, \CC)$--invariant, hence it must be the unique quadric $\mathfrak Q$ containing $\Tan(\Gamma)$. This can be checked with a direct computation which can be left to the reader. In conclusion:

\begin{proposition}\label{prop:hhh} The quadric $\mathfrak Q$ containing $\Tan(\Gamma)$ is such that:\\
\begin{inparaenum}
\item [$\bullet$] its general point is the Hessian of a unique polynomial in $\Sigma(4, 1)$;\\
\item [$\bullet$] its general point is uniquely determined by its Hessian, which also lies in $\mathfrak Q$.
\end{inparaenum}

In particular the { Hessian} map $h_{4,1}$ induces a birational transformation of the quadric $\mathfrak Q$ into itself.
\end{proposition}

\begin{remark}\label{rem:cr}

Given the polynomial \eqref {eq:f}, let $\alpha_1,\ldots, \alpha_4$ be its roots, and {we assume at least three of them are distinct}. We can form the cross ratio $\alpha=(\alpha_1\alpha_2\alpha_3\alpha_4)$.
There is an expression of $\alpha$ which remains unchanged when $\alpha$ is changed in one of the six different values assumed by the cross ratio by permuting $\alpha_1,\ldots, \alpha_4$, namely
\[
J=\frac {4(1-\alpha+\alpha^2)^3}{(\alpha+1)^2(1-2\alpha)^2(2-\alpha)^2},
\]
which is called the \emph{J--invariant} of the { degree 4} divisor of $\PP^1$  determined by the points $\alpha_1,\ldots, \alpha_4$ (see \cite [pp. 27--29]{EnrChi}). It is well known that two divisors of degree 4 are projectively equivalent if and only if they have the same $J$--invariant.

The function $J$ can then be expressed in terms of the coefficients of the polynomial \eqref {eq:f}. Precisely one has
\[
J=\frac 1{3^6\cdot 4^3}\frac {i^3}{j^2}
\]
where
\[
i:=a_0a_4-4a_1a_3 + 3a_2^2
\]
and $j$  is the polynomial defined in \eqref {eq:bu2}. Then $J = 0$ is equivalent to $i = 0$, which is the quadric containing $\Tan(\Gamma)$. When $J = 0$, the degree 4 divisor defined by the vanishing of \eqref{eq:f} is called \emph{anharmonic}, whereas
it is called \emph{harmonic} when $j = 0$, i.e., $J = \infty$.

If we take a general $g_4^1$ on $\PP^1$, defined by an equation of type \eqref {eq:f} with the coefficients $a_i$ depending linearly by $[\lambda,\mu]\in \PP^1$, for $i=0,\ldots,4$, then $J$ varies for the divisors of the $g^1_4$ and for a general value of $J$ there are exactly 6 divisors of the $g^1_4$ for which that value of $J$ is attained. This is no longer the case for the values $J=0$ and $J=\infty$. Indeed, the expression of $J$ tells us that there are only 2 anharmonic divisors in the $g^1_4$ (each counted with multiplicity 3) and 3 harmonic divisors (each counted with multiplicity 2). 

On the other hand there are special $g_4^1$s in which the $J$--invariant is constant. The typical example is the $g_4^1$ defined by the equation $\lambda x_0^4+\mu x_1^4=0$, the general divisor of which is harmonic.
\end{remark}

\begin{remark}\label{rem:syz}
Consider the equation
\begin{equation}\label{eq:syz}
 \mu x_0^4 + 6\lambda x_0^2x_1^2+\mu x_1^4 = 0,\quad\text{ with}\quad  [\lambda, \mu]\in \PP^1.
\end{equation}
 This defines  a $g_4^1$ on $\PP^1$, which is called a \emph{syzygetic} $g_4^1$, as well as any series which is transformed of it via a projective transformation of $\PP^1$.
 
The main property of this series is that it contains the Hessian divisor of every of its divisors.  In fact the Hessian of the divisors defined by \eqref {eq:syz} is easily seen to have equation
\[
\lambda\mu x_0^4+(\mu^2-3\lambda^2)x_0^2x_1^2+\lambda\mu x_1^4=0
\]
which clearly sits in the above $g_4^1$ on which we have the map
\begin{equation}\label{eq:hess}
[\lambda, \mu]\mapsto [\mu^2-3\lambda^2, 6\lambda\mu]
\end{equation}
sending the divisor with equation \eqref {eq:syz} corresponding to $[\lambda, \mu]\in \PP^1$ to its Hessian, i.e., the divisor corresponding to $[\mu^2-3\lambda^2, 6\lambda\mu]$. This correspondence is a $2 : 1$ map. In geometric terms, the line in $\Sigma(4, 1) \cong \PP^4$ corresponding to the above $g_4^1$ is mapped $2 : 1$ by $h_{4,1}$ to another line, rather than to a conic.

Let us compute the $J$--invariant for the divisors in the syzygetic $g_4^1$. One has 
\[
J=\frac {(3\lambda^2+\mu^2)^3} {3^6\cdot 4^3 \lambda ^2(\mu^2-\lambda^2)^2}.
\]
So we see that the $J$--invariant varies inside the syzygetic $g_4^1$. As a consequence, any syzygetic $g_4^1$ is generated by any divisor of degree 4, not in the cone locus, and by its Hessian.

Let us see what are the coincidences of the map \eqref {eq:hess}. We expect three of them, i.e., the intersection of the graph of the $2:1$ map, which is a curve of type $(1,2)$ on $\PP^1\times \PP^1$, with the diagonal, which is a curve of type $(1,1)$.
In fact the coincidences are defined by the degree 3 equation
\begin{equation*}\label{eq:bu22}
\det \left(\begin{matrix} 
\mu^2-3\lambda^2&6\lambda\mu \\
\lambda &\mu\end{matrix}\right)=0,
\end{equation*}
which has the solutions $\mu=0$ and $\mu=\pm 3\lambda$. The corresponding divisors are defined by the equations $x_0^2x_1^2=0$ and $(x_0^2\pm x_1^2)^2=0$, i.e., they consist of two points both with multiplicity 2. We call such a divisor a \emph{bi--double point}: they coincide with their Hessian. 

Each of the bi--double points in the syzygetic $g_4^1$ is also the Hessian of another divisor in the same series.   For instance, the one given by $\mu=0$ also comes form the divisor corresponding to $\lambda=0$, i.e., from the harmonic divisor with equation $x_0^4+x_1^4$. 
Similarly, the divisors given by $\mu=\pm 3\lambda$ come from the divisors corresponding to $[\lambda,\mu]$ such that 
\[
\frac {6\lambda\mu}{\mu^2-3\lambda^2}=\pm 3
\]
that, besides the obvious solution $\mu=\pm 3\lambda$ has also the solution $\mu=\pm \lambda$, which corresponds to the harmonic divisors with equations $x_0^4\pm 6x_0^2x_1^2+x_1^4=0$. In conclusion we have that: \emph{the harmonic degree 4 divisors are characterized by the fact of being the only reduced divisors having a bi--double point as Hessian}.

Next let us consider an anharmonic divisor, with equation $3x_0^4+6x_0^2x_1^2-x_1^4=0$. An easy computation shows that it coincides with the Hessian of its Hessian. Then we consider the composition of the { $2:1$} map defined by \eqref {eq:hess} with itself. It is a  { $4:1$} map, which has 5 coincidences. Three of them are given by the three bi--double points in the 
syzygetic $g_4^1$. Two more are given by the two anharmonic divisors  in the syzygetic $g_4^1$, which are given by the equation \eqref {eq:syz}, with $\mu^2+3\lambda^2=0$. This can be checked with a direct computation that can be left to the reader.

In conclusion we have that: \emph{the anharmonic degree 4 divisors are characterized by the fact of being the only reduced divisors which are Hessian of their Hessian.}
\end{remark}

\begin{remark}\label{rem:AIK} In the paper \cite {AIK}, the authors introduce a map
$$\Phi_{d,r}: \Sigma(d,r)\dasharrow \Sigma((r+1)(d-2),r)$$
which is different from the Hessian map, though it has the same source and target. The map $\Phi_{4,1}$ is again of degree 2. So, as the Hessian map, it determines a birational involution of $\Sigma(4,1)$. According to \cite [Theorem 3.2]{AIK}, this involution is uniquely determined by the property of being equivariant under the natural ${\rm SL}(2,\CC)$ action. Therefore, though $\Phi_{4,1}$ and $h_{4,1}$ are different, they determine the same birational involution on $\Sigma(4,1)$. This means that there is a non--trivial birational map $\phi$ of $\Sigma(4,1)$ such that $\Phi_{4,1}=\phi\circ h_{4,1}$. 
\end{remark}

\section{The case $r=2, d=3$}\label {sec:2,3}

In this section we examine the map $h_{3,2}: 
\Sigma(3,2)\cong \PP^9\to \Sigma(3,2)\cong \PP^9$. This is the only other case, besides $h_{4,1}$, in which the domain and the target of the Hessian map coincide. Most of the results of this section are based on \cite {DK} and on  \cite[Vol. 2, Chapt. III] {EnrChi} and are essentially known in the current literature, although we put them here in our general perspective. 

Let us consider the Veronese surface $V_{3,2}$ in $\Sigma(3,2)\cong \PP^9$, which is the locus of curves of the type $3L$, with $L$ a line of the plane. 

\begin{lemma}\label{lem:triv} The cone locus $\C_{3,2}$ coincides with $\Sec(V_{3,2})$.
\end{lemma}

\begin{proof} It is immediate that $\Sec(V_{3,2})$ is contained in $\C_{3,2}$. Moreover they are both irreducible, of dimension 5 (because $V_{3,2}$ is not defective). This proves the assertion.  \end{proof}

Recall  that $h_{3,2}$ is defined by a linear system $\H_{3,2}$ of cubics containing $\C_{3,2}$.

\begin{proposition}\label{prop:ttt}  $\H_{3,2}$ is a linear system of cubics singular  along $V_{3,2}$.
\end{proposition}
\begin{proof} All cubics in $\H_{3,2}$ contain $\C_{3,2}=\Sec(V_{3,2})$. On the other hand 
$\Sec(V_{3,2})$ is singular along $V_{3,2}$. Moreover, if $x$ is a point of $V_{3,2}$, the tangent cone to $\Sec(V_{3,2})$ at $x$ is the cone over $V_{3,2}$ with vertex the tangent plane to $V_{3,2}$ at $x$ (see \cite [Thm. 3.1]{CilRu}), and this cone is non--degenerate in $\PP^9$. This implies that a cubic containing $\Sec(V_{3,2})$ is singular at any point of $V_{3,2}$.\end{proof}

\begin{proposition}\label{prop:GGNN}
The Gordan--Noether scheme structure on the cone locus  $\C_{3,2}$ is non--reduced. 
\end{proposition}
\begin{proof}
A computation (which can be left to the reader) of the tangent space to $\GN_{3,2}$ at the point corresponding to $x_0^3+x_1^3$,
analogous to the proof of Proposition \ref{prop:GN},
shows it has codimension $3$, smaller than $4$, which is the codimension of the cone locus $\C_{3,2}$ in $\Sigma(3,2)\cong\PP^9$. \end{proof}

\begin{remark}\label{rem:gol}  The orbit of $x_0^3+x_1^3$ via the  $\SL(3,\CC)$ action is dense in $\C_{3,2}$, thus $\GN_{3,2}$ is non--reduced all along $\C_{3,2}$. Actually, 
a Macaulay2 \cite{M2} computation shows that  $\GN_{3,2}$ has degree $30$, which is the double of the degree of the cone locus. This tells us that $\GN_{3,2}$ is a double structure on $\C_{3,2}$. \end{remark}

\begin{remark} We have checked in a similar way that $\GN_{3,3}$ is non--reduced at a general point of $\C_{3,3}$,
and $\GN_{3,4}$ is non--reduced at a general point of $\C_{3,4}$. On the contrary,
 $\GN_{3,4}$ is reduced at the point corresponding to the Perazzo cubic threefold recalled in the Introduction, so that its
$\SL(5)$-orbit makes a $18$-dimensional reduced component of $\GN_{3,4}$.\end{remark}

Consider now $\Sec_2(V_{3,2})$ the variety of 3--secant planes to $V_{3,2}$, which is a hypersurface in $\Sigma(3,2)\cong\PP^9$, called the \emph{Aronhold hypersurface}. Any trisecant plane to $V_{3,2}$ is cut out by $\H_{3,2}$ in a fixed cubic, namely the union of the three lines pairwise joining the three points of intersection of the plane with $V_{3,2}$. Thus this plane is contracted to a point by $h_{3,2}$. Precisely, up to a projective transformation, we may assume that the three points of $V_{3,2}$ are the triple lines $x_0^3=0$, $x_1^3=0$, $x_2^3=0$. Hence the plane spanned by them parameterizes all cubics of the form $ax_0^3+bx_1^3+cx_2^3=0$. The Hessian of the general such cubic has equation $x_0x_1x_2=0$, i.e., it is the trilateral union of the three original (independent) lines. In conclusion the hypersurface $\Sec_2(V_{3,2})$ is contracted by $h_{3,2}$ to the 6--dimensional variety, isomorphic to the triple symmetric product of $\Sigma(1,2)$, the dual of $\PP^2$, whose general point is a trilateral. 

 { \begin{remark}\label{rem:res} Taking into account Proposition  \ref {prop:GGNN}, we understand that the resolution of the indeterminacies of $h_{3,2}$ is complicated. However it is possible to see what happens in some specific cases.

Consider a general point of  $\Sec(V_{3,2})$, the cone locus. Up to a change of coordinates, we may assume that this point coincides with the triple of lines with equation $\alpha=x_0x_1(x_0+x_1)=0$. Let $f(x_0,x_1,x_2)$ be a general homogeneous polynomial of degree 3. Consider the pencil of cubics with equation $\alpha+tf=0$, with $t\in \CC$. The limit of $\hess(\alpha+tf)$, when $t\to 0$, is easily seen to be 
$-4f_{22} (x_0^2-x_0x_1-x_1^2)$, namely a suitable trilateral.

However one has to be careful. Indeed, if $\alpha=x_0^3$ and $f=x_1^3+x_2^3+36x_0x_1x_2$, the limit of $\hess(\alpha+tf)$, when $t\to 0$, 
is $x_0(x_1x_2-36x_0^2)$, which splits as a smooth conic plus a secant line.
\end{remark} }

\begin{thm}\label{thm:im} The Hessian map $h_{3,2}$ is dominant and generically $3:1$.
\end{thm}

\begin{proof} It is well known that every smooth plane cubic is $\SL(3)$-equivalent to a member of the pencil
\begin{equation}\label{eq:stanfinite}
x_0^3+x_1^3+x_2^3-3tx_0x_1x_2=0, \quad \text{for some}\quad t\in \CC.
\end{equation}
For a modern reference see \cite[Lemma 1]{ArtDolg}, this pencil is classically called the \emph{syzygetic pencil} or \emph{Hesse pencil},
see next Proposition \ref{prop:sssig}.
A direct computation shows that the Hessian of the cubic with equation \eqref {eq:stanfinite} is the cubic with equation
\[
x_0^3+x_1^3+x_2^3-3sx_0x_1x_2=0, \quad \text{with}\quad s=\frac {4-t^3}{3t^2}.
\]
Since the Hessian map is $\SL(3)$-equivariant, the result follows.
\end{proof} 

Next a few words about the $J$--invariant of a cubic. Consider a smooth cubic $F\subset \PP^2$ with equation
\[
f(x_0,x_1,x_2)=\sum_{i+j+k=3}a_{ijk}x_0^ix_1^jx_2^k=0.
\]
We will denote by $\bf a$ the vector of the coefficients of the polynomial $f$, lexicographically ordered. { The classical} \emph{Salmon's theorem} {(see \cite [p. 189]{EnrChi2})} says that given a general point $p\in F$, the $J$--invariant of the four tangents to $F$ through $p$ (different from the tangent at $p$), does not depend on $p$. It is called the \emph{$J$--invariant}, or \emph {modulus}, of the cubic $F$ and two cubics are projectively equivalent if and only if they have the same modulus. 
A cubic is called \emph{harmonic} [resp. \emph {anharmonic}] if $J=\infty$ [resp. if $J=0$]. The concept of $J$--invariant  can be extended to nodal cubics in which case $J=1$. 

We record the following theorem (see \cite [p. 199]{EnrChi2}):

\begin{thm}[Aronhold's Theorem]\label{thm: aron} A cubic is anharmonic if and only if it is projectively equivalent to the \emph{Fermat cubic} with equation
\[
x_0^3+x_1^3+x_2^3=0.
\]
\end{thm}

The $J$--invariant is a homogeneous rational function of $\bf a$ and it is well known (see \cite [Chapt. III, \S 25]{EnrChi2}), that there are two homogeneous polynomials $S({\bf a})$ and $T({\bf a})$, the former of degree 4 called the \emph{Aronhold invariant} (see \cite{Ott09} for a Pfaffian presentation), the latter of degree 6, such that 
\[
J=\frac {S^3}{T^2}. 
\]
This implies the well known fact that in a general pencil of cubics there are 12 curves with a fixed value of $J$ (in particular there are 12 singular cubic, to be counted with the appropriate multiplicity), except for the anharmonic curves, of which there are 4 each counted with multiplicity 3, and for the harmonic ones, of which there are 6 each to be counted twice. For special pencils of cubics the modulus can be constant: these pencils are classified in \cite{Chi} and it turns out that the singular curves in these pencils are not nodal.  

Going back to the syzygetic pencil, it may be defined as the the pencil {generated by a smooth cubic curve $F$ and by its Hessian}, as well as any pencil which is projectively equivalent to it.  It is well known that the intersection of $F$ and $\Hess(F)$ consists of 9 distinct points, which are the flexes of $F$. We note that, by the configuration of flexes of a cubic there are four trilaterals with the property that each of them contains all flexes of $F$ (see \cite [Vol. 2, p. 214] {EnrChi}). Each of these trilaterals sits in the syzygetic pencil, and the four of them account for the 12 singular cubics in the syzygetic pencil. 

Since the flexes of a plane curve are characterized by the property of being the smooth points  of the curve which are the intersections of the curve with its Hessian, if a cubic contains a line $L$, then the Hessian also contains $L$, because all the points of $L$ are flexes of the cubic. In particular, the Hessian of a trilateral is the same trilateral, what can be proved also with a direct computation, assuming, as we can, that the trilateral has equation $x_0x_1x_2=0$.

\begin{thm}[Hesse's Theorem, \cite {EnrChi}, p. 214]\label{thm:hesses} All curves of the { syzygetic pencil} are smooth at the nine base points of the pencil which are flexes for all of them.
\end{thm}

We note that the {curves in the syzygetic pencil do not have constant modulus}, because the singular curves in {the pencil} are all nodal. Hence we can find some anharmonic cubic in { the syzygetic pencil}.  By Theorem \ref {thm: aron}, up to projective transformation we may assume that such an anharmonic cubic is the Fermat cubic, hence we conclude with the:

\begin{proposition}\label{prop:sssig} The syzygetic pencils are all projectively equivalent to  the pencil with equation
\begin{equation}\label{eq:stan}
x_0^3+x_1^3+x_2^3-3tx_0x_1x_2=0, \quad \text{with}\quad t\in \CC\cup \{\infty\}.
\end{equation}
\end{proposition}

We saw in the proof of Theorem \ref{thm:im} that the Hessian map behaves on the pencil \eqref{eq:stan}
as the map \begin{equation}\label{eq:hessts}
t\in \CC\cup \{\infty\} \to s=\frac {4-t^3}{3t^2}\in \CC\cup \{\infty\},
\end{equation}
the extension to $t=\infty$ being immediate.

There are here 4 coincidences. Since these are cubics which coincide with their Hessian, all of their (smooth) points  are flexes, hence they are the 4 trilaterals contained in the pencil. 
One cleary corresponds to the value $t=\infty$, i.e., it is the trilateral $x_0x_1x_2=0$. The other three are obtained solving the equation
\[
t=\frac {4-t^3}{3t^2}
\]
hence they correspond to the three roots of unity $t=1, \epsilon=\exp \frac {2\pi i}3, \epsilon ^2$. These trilaterals are also Hessian of some other cubic in the pencil, which we now compute. First, for $s=\infty$, we find the equation $t^2=0$, hence we find the anharmonic cubic $x_0^3+x_1^3+x_2^3=0$ with multiplicity 2. For $s=1$, we find the equation $t^3+3t^2-4=0$, which has the obvious solution  $t=1$. Dividing by $t-1$ we get $t^2+4t+4=0$, hence again a double root $t=-2$. Similarly, for $t=\epsilon, \epsilon^2$, we find double roots $t=-2\epsilon, t=-2\epsilon^2$. It is easy to see that these four values are exactly the ones where the map  \eqref {eq:hessts} ramifies. Hence we have:

\begin{proposition}\label{prop:off} One has:\\
\begin{inparaenum}
\item [(i)] the four trilaterals in the syzygetic pencil are self Hessian and are also Hessian of only one other cubic of the pencil, each to be counted with multiplicity 2;\\
\item [(ii)] these are the four anharmonic cubics of the pencil;\\
\item [(iii)] the anharmonic cubics are characterized by the fact that their Hessian is a trilateral;\\
\item [(iv)] the Hessian of a smooth plane cubic which is not anharmonic is a smooth curve. 
\end{inparaenum}
\end{proposition}
\begin{proof} The assertion (i) is clear. The  assertion (ii) is also clear for the curve $x_0^3+x_1^3+x_2^3=0$ corresponding to $t=0$. Then it follows in the other three cases because the anharmonic cubics are all projective, so the Hessian of any anharmonic cubic is a trilateral. Assertion (iii) is an immediate consequence of the above arguments. As for assertion (iv), it follows from the fact that the only singular curves in the syzygetic pencil are the four trilaterals.   \end{proof}

Next we prove the:

\begin{lemma}\label{lem:arm} An harmonic cubic is the Hessian of its Hessian.
\end{lemma} 

\begin{proof} The harmonic cubics are all projectively equivalent. Hence we can consider the cubic with equation
\[
f(x_0,x_1,x_2)=x_1^3-x_2^2x_0-px_1x_0^2=0, \quad \text{with}\quad p\neq 0.
\]
The reader will check that this curve is harmonic and that $\hess(\hess(f))=8^3\cdot 6^3\cdot p^2\cdot f$, proving the assertion. \end{proof}

Next we consider the composition of the map \eqref {eq:hessts} with itself, which is a $(9:1)$ map. Then it has 10 coincidences. Four of them are again the four trilaterals in the syzygetic pencil. The remaining six are the six harmonic cubics in the pencil which, by Lemma \ref {lem:arm} are Hessian of their Hessian. This proves that:

\begin{proposition}\label{prop:harm} The harmonic cubics are characterized by the fact of being the Hessian of their Hessian.
\end{proposition}

The six harmonic cubics in the syzygetic pencil can be explicitely computed. Indeed, one is led to solve the equation
\[
t=\frac {4-\Big (\frac {4-t^3}{3t^2}\Big)^3}{3 \Big ( \frac {4-t^3} {3t^2} \Big)^2}
\]
which has degree 9. Dividing by $t^3-1$ (corresponding to the three trilaterals different from { $x_0x_1x_2=0$} which corresponds to $t=\infty$), we find the degree 6 equation
\[
t^6-20t^3-8=0
\]
whose solutions correspond to the harmonic cubics in the pencil. 

Finally we have the following Lemma (compare with the rough classification in \cite[Table 2]{Banchi}):

\begin{lemma}\label{lem:casi} One has:\\
\begin{inparaenum}
\item [(i)] the Hessian of a smooth cubic is a smooth cubic, except in the anharmonic case in which it is a trilateral;\\
\item [(ii)] the Hessian of a cubic with a node is a cubic with the same node and the same tangent lines at the node;\\
\item [(iii)] the Hessian of a cubic with a cusp is the cuspidal tangent counted with multiplicity 2 plus the line joining the cusp with the only flex of the curve;\\
\item [(iv)] the Hessian of a cubic reducible in a conic $\Gamma$ plus a line $L$ which is not tangent to $\Gamma$,  consists of $L$ plus a conic which is tangent to $\Gamma$ at the points where $L$ intersects $\Gamma$;\\
\item [(v)] the Hessian of a cubic reducible in a conic $\Gamma$ plus a line $L$ which is  tangent to $\Gamma$,  consists of $L$ counted with multiplicity 3;\\
\item [(vi)] the Hessian of a trilateral is the same trilateral;\\
\item [(vii)] the Hessian is undetermined for the cones (i.e., triple of lines concurrent at a point).
\end{inparaenum}
\end{lemma}

\begin{proof} Part (i) is Proposition \ref {prop:off}. The rest of the assertion can be proved with explicit computations which can be left to the reader (take into account that nodal cubics and cuspidal cubics are all projectively equivalent).  
\end{proof}

In conclusion, we have the:

\begin{proposition}\label{prop:in} The following cubics are the only ones which do not appear as Hessian of some other cubic:\\
\begin{inparaenum}
\item [(i)] cubics with a cusp;\\
\item [(ii)] cubics reducible in a conic plus a line  tangent to the conic;\\
\item [(iii)] cubics reducible in three distinct lines passing through a point.
\end{inparaenum}
\end{proposition}

On the other hand, the cubics listed in Proposition \ref {prop:in} do appear in the image of the resolution of the indeterminacies of $h_{3,2}$. It would be nice to understand in details how this works. 

\begin{remark}\label{rem:aron} The equation $S=0$, $S$ being the Aronhold invariant, defines in $\Sigma(2,3)$ a hypersurface which is the Aronhold hypersurface $\Sec_2(V_{3,2})$. This is singular along $\Sec (V_{3,2})$. Hence, the polar map of the Aronhold hypersurface
\[
{\rm ar}: \Sigma(2,3)\dasharrow \Sigma(2,3)
\]
is also defined by a 9--dimensional $\SL(3)$-invariant linear system of cubics containing $\Sec(V_{3,2})$, like the Hessian map. We call it the \emph{Aronhold map}. It is remarkable that this has also degree 3 but it is different from the Hessian map. The action of it on the syzygetic pencil \eqref {eq:stan}, sends the curve corresponding to the parameter $t$ to the curve corresponding to the parameter $\frac {2+t^3}{3t}$. We verified this using Macaulay2. The composition ${\rm ar}^2$ has exactly the same fixed points found in Proposition \ref{prop:harm}, namely the four trilaterals and the six harmonic cubics.
\end{remark}

In conclusion, we want to remark that $H^0(\PP^9, \mathcal I_{\Sec(V_{3,2}),\PP^9}(3))$ is a $20$--dimensional representation of $\SL(3,\CC)$, which splits in the sum of two $10$--dimensional representations, one, isomorphic to $\Sym^3(V)$, is the vector space of cubics corresponding to the linear system $\mathcal H_{3,2}$, the other is the $10$--dimensional vector space of polars of the Aronhold invariant $S$. It is interesting to notice that:

\begin{proposition}\label{prop:sceme}
The polars of the  Aronhold invariant cut out the cone locus schematically.
\end{proposition}
\begin{proof}
The jacobian of the polar map of the Aronhold invariant coincides with the $10\times 10$ Hessian matrix of the Aronhold invariant. 
Evaluating at $x_0^3+x_1^3$, one finds
\[\begin{pmatrix}
      0&0&0&0&0&0&0&0&0&0\\
      0&0&0&0&0&0&0&0&0&0\\
      0&0&0&0&0&0&0&0&0&0\\
      0&0&0&0&0&0&0&0&0&0\\
      0&0&0&0&0&0&0&0&0&{-216}\\
      0&0&0&0&0&0&0&0&144&0\\
      0&0&0&0&0&0&0&0&0&0\\
      0&0&0&0&0&0&0&0&0&0\\
      0&0&0&0&0&144&0&0&0&0\\
      0&0&0&0&{-216}&0&0&0&0&0\end{pmatrix}\]
which has rank $4$, equal to the codimension of the cone locus. \end{proof}

\section{Hypersurfaces of rank $r+2$}\label{sec:r+2}

Recall that the \emph{Waring rank}, or simply the \emph{rank}, of a polynomial $f\in \Sym^d(V)$, or of the polynomial class $[f]\in \Sigma(d,r)$, is the minimum integer $h$ such that $[f]$ sits on a linear space of dimension $h-1$ which is $h$--secant to the Veronese variety $V_{d,r}\subset \Sigma(d,r)$. This is the same as saying that $h$ is the minimum such that $f$ can be written as
\[
f=l_1^d+\cdots+l_h^d,
\]
with $l_1,\ldots, l_h$ non--proportional linear forms. 

In this section we study the Hessian of polynomials $f$ of rank $r+2$, with $[f]\in \Sigma(d,r)$. First we compute the Hessian of a polynomial of rank $r+2$.

\begin{proposition}\label{prop:expr} Let $f=\sum_{i=0}^{r+1}c_i l_i^d$, where $l_0,\ldots,l_{r+1}$ are linear forms. Then, up to a scalar, one has
 \[\hess(f)=\sum_{i=0}^{r+1}\prod_{j\neq i}c_jl_j^{d-2}.\] 

\end{proposition}
\begin{proof} We have to prove a polynomial identity. So it is sufficient to assume that the forms $l_0,\ldots,l_{r+1}$ are  general under the condition $\sum_{i=0}^{r+1} {l_i}=0$ and $c_i\neq 0$ for $i=0,\ldots,r+1$. 

We consider in $\PP^{r+1}$, with homogeneous coordinates $[l_0,\ldots, l_{r+1}]$,  the hypersurface $\Phi$ defined by the equation $\sum_{i=0}^{r+1} c_i l_i^d=0$ and the hyperplane $H$ with equation $\sum_{i=0}^{r+1} l_i=0$.
The Hessian of the intersection hypersurface of $H$ with $\Phi$ is the locus of points $p=[p_0,\ldots, p_{r+1}]\in H$ such that the polar quadric 
$\sum_{i=0}^{r+1} c_i p_i^{d-2} l_i^{2}=0$ to $\Phi$ with respect to $p$ is tangent to $H$. That is the dual quadric $\sum_{i=0}^{r+1} \frac{1}{c_i p_i^{d-2}}l_i^{2}=0$ contains the point $[1,\ldots, 1]$,
namely $\sum_{i=0}^{r+1} \frac{1}{c_i p_i^{d-2}}=0$. Getting rid of the denominators, we find the desired equation.
\end{proof}

As a consequence we have that a polynomial of rank $r+2$ can be recovered from its Hessian. This is an immediate consequence of the following:

\begin{proposition}\label{prop:ciror} Let $f=\sum_{i=0}^{r+1}c_i l_i^d$, $g=\sum_{i=0}^{r+1}b_i l_i^d$, where $l_0,\ldots,l_{r+1}$ are linear forms such that any $r+1$ of them are linearly independent and such that $\prod_{i=0}^{r+1}c_i$ and $\prod_{i=0}^{r+1}b_i$ are both nonzero. Suppose that $\hess(f)$ and $\hess(g)$ are proportional. Then
$(c_0,\ldots c_{r+1})$ is proportional to $(b_0,\ldots b_{r+1})$.
\end{proposition}

\begin{proof} We may assume $l_i=x_i$, for $i=0,\ldots, r$, and $l_{r+1}=x_0+\ldots +x_r$.

The Hessian of $f$, up to a factor, is 
\begin{equation}\label{eq:eq1}
\sum_{i=0}^r G_i(c)\frac{\prod_{j=0}^r x_j^{d-2}}{x_i^{d-2}}(x_0+\ldots +x_r)^{d-2}+G_{r+1}(c)(\prod_{j=0}^r x_j^{d-2})
\end{equation}
where $G_i(c)=\frac{1}{c_i}\prod_{j=0}^{r+1} c_j$ for $i=0,\ldots ,r+1$. Similarly, the Hessian of $g$ is
 \begin{equation}\label{eq:eq2}
 \sum_{i=0}^r G_i(b)\frac{\prod_{j=0}^r x_j^{d-2}}{x_i^{d-2}}(x_0+\ldots +x_r)^{d-2}+G_{r+1}(b)(\prod_{j=0}^r x_j^{d-2}).
 \end{equation}
Multiplying (\ref{eq:eq1}) by $G_0(b)$,  (\ref{eq:eq2}) by $G_0(c)$ and subtracting we get
$$x_0^{d-2}\left(\sum_{i=1}^r A_i(b,c)\frac{\prod_{j=0}^r x_j^{d-2}}{x_i^{d-2}}(x_0+\ldots +x_r)^{d-2}+A_{r+1}(b,c)(\prod_{j=1}^r x_j^{d-2})\right)$$
where $A_i(b,c)=G_0(b)G_i(c)-G_0(c)G_i(b)$ for $i=0,\ldots, r+1$.
Since we assumed $\hess(f)$ and $\hess(g)$ proportional, the last equation either equals again $\hess(f)$ (or $\hess(g)$) up to a constant, or it is identically zero. Since the Hessian is not divisible by $x_0$ we get that
the  polynomial 
\begin{equation}\label{eq:eq3}
\sum_{i=1}^r A_i(b,c)\frac{\prod_{j=0}^r x_j^{d-2}}{x_i^{d-2}}(x_0+\ldots +x_r)^{d-2}+A_{r+1}(b,c)(\prod_{j=1}^r x_j^{d-2})
\end{equation} 
vanishes identically. There is no term in (\ref{eq:eq3}) in which $x_1^h$ appears with $h>2(d-2)$. The term in which $x_1^{2(d-2)}$ appears has coefficient
$\sum_{i=2}^r A_i(b,c)\frac{\prod_{j=0}^r x_j^{d-2}}{(x_1x_i)^{d-2}}$  and this yields $A_i(b,c)=0$ for $i=2,\ldots, r$.
In the same way $A_1(b,c)=0$, which in turn implies $A_{r+1}(b,c)=0$. We get
$$\textrm{rk}\begin{pmatrix}G_0(b)&G_1(b)&\ldots&G_{r+1}(b)\\
G_0(c)&G_1(c)&\ldots&G_{r+1}(c)\end{pmatrix}=1,$$ which is easily seen to imply
$$\textrm{rk}\begin{pmatrix}b_0&b_1&\ldots&b_{r+1}\\
c_0&c_1&\ldots&c_{r+1}\end{pmatrix}=1.$$
\end{proof}

\begin{remark} The Hessian of $\sum_{i=0}^{r}c_i x_i^d$ (where $c_{r+1}=0$) is $\prod _{i=0}^{r}x_i^{d-2}$
and does not depend on $c_i$. This shows that the assumption in Proposition \ref{prop:ciror}  that $\prod_{i=0}^{r+1}c_i$ is nonzero cannot be removed.
\end{remark}

If $f=\sum_{i=0}^{r+1}c_i l_i^d$ is, as above, a polynomial of rank $r+2$, and $F$ is the hypersurface $f=0$, then the combinatorics of $\mathrm{Sing}(\Hess(F))$ is interesting and rich, as shown by the following:

\begin{proposition}\label{prop:singhess} Let $f=\sum_{i=0}^{r+1} l_i^d$ be general a polynomial of rank $r+2$, and $F$ the hypersurface $f=0$. Then $\Hess(F)$ has:\\
\begin{inparaenum}
\item [$\bullet$] multiplicity $d-2$ at the general point of each of  the ${{r+2}\choose 2}$ codimension two linear subspaces $L_{ij}$ with equations $l_i=l_j=0$, with $0\leq i<j\leq r+1$;\\
\item  [$\bullet$] multiplicity $2(d-2)$ at the general point of each of the ${{r+2}\choose 3}$ codimension three linear subspaces $L_{ijk}$ with equations $l_i=l_j=l_k=0$, with $0\leq i<j<k\leq r+1$.
\end{inparaenum} 

Moreover, if $d\geq 4$ these are the only singularities of $\Hess(F)$ in codimension 1 and if $d=3$ these are the only singularities in codimension $c\leq 2$.

In particular, from the configuration of the singularities of $\Hess(F)$ one can uniquely recover the linear forms $l_0,\ldots, l_{r+1}$ up to a factor. 
\end{proposition}
\begin{proof}  We recall Proposition \ref {prop:expr}, which shows that $\Hess(F)$ has equation 
\begin{equation}\label{eq:port}
\sum_{i=0}^{r+1}\prod_{j\neq i}l_j^{d-2}=0. 
\end{equation}
Each summand in \eqref {eq:port} has multiplicity at least $d-2$ along the subspaces $L_{ij}$, with $0\leq i<j\leq r+1$, and it has multiplicity at least $2(d-2)$ along the subspaces  $L_{ijk}$, with $0\leq i<j<k\leq r+1$. To see that these are the exact multiplicities, intersect $\Hess(F)$ with a hyperplane $l_i=0$, with $i=0,\ldots, r+1$. The intersection $L_i$ has equations
\[
l_i=0, \quad \prod_{j\neq i}l_j^{d-2}=0
\]
and this shows that the multiplicities are as in the statement. 

Next, let us prove the assertion about the dimension of the singular locus. Assume first $d\geq 4$ and  suppose there are other singularities in codimension 1. This would force the intersections $L_i$ to have multiplicity worse than $d-2$ at some point  of a space $L_{ij}$, with $j\neq i$, off the spaces $L_{ijk}$, with $i\neq k\neq j$, or to have  multiplicity worse than $2(d-2)$ along some of the spaces $L_{ijk}$, with $i,j,k$ distinct. Since this is not the case, the assertion is proved. 

Finally, consider the case $d=3$. Suppose first that there are no singularities in codimension 1. Since the hypersurface in $\PP^r$ defined by the vanishing  of a general symmetric determinant of order $r+1$ of linear forms has a singular locus of multiplicity 2, pure codimension 2 and degree ${{r+2}\choose 3}$ (see \cite [\S 4]{Segre}), this implies that there are no other singularities besides the ${{r+2}\choose 3}$ subspaces $L_{ijk}$, with $0\leq i<j<k\leq r+1$. 

In conclusion we have to prove that there are no singularities in codimension 1. Suppose by contradiction this is not the case and let $\Sigma$ be the codimension one singular locus of $\Hess(F)$. By looking at the equation of the sections $L_i$ we considered above, we see that the only possibility is that $\Sigma$ intersects $L_i$ along some of the spaces $L_{ijk}$, with $i,j,k$ distinct. Since we are dealing with a general polynomial $f$ of rank $r+2$, we may assume that there is a monodromy action on the polynomials $l_i$, which acts as the full symmetric group $\mathfrak S_{r+2}$. This implies that $\Sigma$ should equally cut $L_i$ along all the spaces $L_{ijk}$, with $i,j,k$ distinct, which are in number of ${{r+1}\choose 2}$. Hence $\deg (\Sigma)= {{r+1}\choose 2}$. On the other hand $\Hess(F)$ has degree $r+1$, hence its general plane curve section would be a curve of degree $r+1$ with ${{r+1}\choose 2}$ singularities, hence it would be the union of $r+1$ lines. This would imply that $\Hess(F)$ splits as the sum of $r+1$ hyperplanes which contradicts (\ref{eq:port}).
\end{proof}

As an immediate consequence we have:

\begin{thm}\label{prop:r+2} Let $r\geq 2$ and $d\geq 3$.  If $f$ and $g$ are general polynomials of degree $d$ and rank $r+2$ in $\PP^r$ such that $\hess(f)$ and $\hess(g)$ are proportional, then $f$ and $g$ also differ by a scalar multiple. 

In other words, the restriction of the hessian map $h_{d,r}$ to the locus of hypersurfaces of rank $r+2$ in $\Sigma(d,r)$ is birational onto its image. 
\end{thm}
\begin{proof} Let $f=\sum_{i=0}^{r+1}c_i l_i^d$ and let $F$ be the hypersurface $f=0$. By Proposition \ref {prop:singhess}, we have that from $\Sing(\Hess(F))$ we can uniquely recover the linear forms $l_i$, $i=0,\ldots, r+1$, up to a factor. Hence we have  $g=\sum_{i=0}^r b_il_i^d$. Then Proposition \ref{prop:ciror} shows that $f$ and $g$ differ by a scalar. \end{proof}

 As a further consequence we have the:

\begin{thm}\label{thm:sylv} The Hessian map
\[
h_{3,3}: \Sigma(3,3)\dasharrow \Sigma(4,3)
\]
is birational onto its image. Hence the general cubic surface $F$ is uniquely determined by its Hessian $\Hess(F)$.  
\end{thm}

This follows right away from Theorem \ref {prop:r+2}  and from the famous:

\begin{thm}[Sylvester Pentahedral Theorem, see \cite {ShB}] \label{thm:penta} There is a dense Zariski subset $U$ of $\Sigma(3,3)$ such that every $[f]\in U$ can be written as
\begin{equation}\label{eq:sylv}
f=l_0^3+l_1^3+l_2^3+l_3^3+l_4^3
\end{equation}
with $l_0,\ldots, l_4$ linear forms, which are uniquely determined up to permutation and numerical factors which are cubic roots of the unity.

In other terms the general $[f]\in \Sigma(3,3)$ has rank 5. 
\end{thm}

\section{Generic finiteness of the Hessian map}\label{sec:finite}

In this section we study the infinitesimal behaviour of the Hessian map at a general rank  $r+2$ hypersurface. In doing this we will prove that $h_{d,r}$ is generically finite onto its image for $r\geq 2$ and $d\geq 3$. 
 
\begin{thm}\label{thm:der_hessian}
  Let $l_i=x_i$ for $i=0,\ldots, r$, $l_{r+1}=\sum_{i=0}^rx_i$. Let $f=\sum_{i=0}^{r+1}l_i^d$, which is a general polynomial of rank $r+2$. Let $d\geq 3$,
$r\geq 2$. The differential of the hessian map $h_{d,r}$ at $[f]$ is injective. 
\end{thm}

Remember that $\PP^r=\PP(V^\vee)$, where we introduced { homogeneous} coordinates $[x_0,\ldots, x_r]$. We consider the Hessian map
\[
\hess_{d,r}: f\in \Sym^d(V)\to \hess(f)\in \Sym^{(r+1)(d-2)}(V)
\]
at the polynomial level, and set $H:=\hess_{d,r}$. Theorem \ref {thm:der_hessian} follows from the following Lemma.

\begin{lemma} With notation and assumptions as in Theorem \ref{thm:der_hessian}, we have:\\
\begin{inparaenum} 
\item [(i)]  the differential of $H$ at $f$ is
$$dH_f(g)=\sum_{i=0}^rl_0^{d-2}\ldots\widehat{l_i^{d-2}}\ldots l_r^{d-2}{\partial_i^2}{g}+\sum_{0\le i<j\le r}l_0^{d-2}\ldots \widehat{l_i^{d-2}}\ldots \widehat{l_j^{d-2}}\ldots l_{r+1}^{d-2}{\left(\partial_i-\partial_j\right)^2}{g}$$ 
for any $g\in  T_f({\Sym}^d(V))$, where $\partial_i$ stays for the derivative with respect to $x_i$ and $(\partial_i-\partial_j)^2$ is the symbolic power;\\
\item [(ii)]  $dH_f$ is injective.
  \end{inparaenum}
\end{lemma}

\begin{proof}
  The Hessian $H(f)$ is, up to a constant factor,  the determinant of the matrix
  $$\begin{pmatrix}l_0^{d-2}+l_{r+1}^{d-2}&l_{r+1}^{d-2}&\ldots&l_{r+1}^{d-2}\\
     l_{r+1}^{d-2}&l_1^{d-2}+l_{r+1}^{d-2}&\ldots&l_{r+1}^{d-2}\\
       \vdots&&\ddots&\vdots\\
    l_{r+1}^{d-2}&l_{r+1}^{d-2}&\ldots&l_r^{d-2}+l_{r+1}^{d-2}\end{pmatrix}.$$
  Then $dH_f(g)$ equals
  $$\begin{vmatrix}g_{00}&g_{01}&\ldots&g_{0r}\\
    l_{r+1}^{d-2}&l_1^{d-2}+l_{r+1}^{d-2}&\ldots&l_{r+1}^{d-2}\\
    \vdots&&\ddots&\vdots\\
    l_{r+1}^{d-2}&l_{r+1}^{d-2}&\ldots&l_r^{d-2}+l_{r+1}^{d-2}\end{vmatrix}+
\begin{vmatrix}l_0^{d-2}+l_{r+1}^{d-2}&l_{r+1}^{d-2}&\ldots&l_{r+1}^{d-2}\\
    g_{01}&g_{11}&\ldots&g_{1r}\\
    \vdots&&\ddots&\vdots\\
    l_{r+1}^{d-2}&l_{r+1}^{d-2}&\ldots&l_r^{d-2}+l_{r+1}^{d-2}\end{vmatrix}+\ldots$$
$$+
  \begin{vmatrix}l_0^{d-2}+l_{r+1}^{d-2}&l_{r+1}^{d-2}&\ldots&l_{r+1}^{d-2}\\
    l_{r+1}^{d-2}&l_1^{d-2}+l_{r+1}^{d-2}&\ldots&l_{r+1}^{d-2}\\
    \vdots&&\ddots&\vdots\\
    g_{0r}&g_{1r}&\ldots&g_{rr}\end{vmatrix}$$
  Expand this sum collecting the monomials in $l_i$. The monomial
  $l_0^{d-2}\ldots\widehat{l_i^{d-2}}\ldots l_r^{d-2}$ appears with coefficient $g_{ii}$: here only the diagonal term in the $i$-th summand is involved.
  The monomial $l_0^{d-2}\ldots \widehat{l_i^{d-2}}\ldots \widehat{l_j^{d-2}}\ldots l_{r+1}^{d-2}$
  appears with coefficient involving both the $i$-th and the $j$-th summand, where we have respectively the two minors
  \[\begin{pmatrix}g_{ii}&g_{ij}\\l_{r+1}^{d-2}&l_j^{d-2}+l_{r+1}^{d-2}\end{pmatrix}\quad \text{ and}\quad 
  \begin{pmatrix}l_i^{d-2}+l_{r+1}^{d-2}&l_{r+1}^{d-2}\\g_{ij}&g_{jj}\end{pmatrix}.\]
  The corresponding coefficients are respectively $l_{r+1}^{d-2}(g_{ii}-g_{ij})$ and $l_{r+1}^{d-2}(g_{jj}-g_{ij})$, the resulting sum is  $l_{r+1}^{d-2}\left(\partial_i-\partial_j\right)^2g$. This proves (i). 
  
In order to prove (ii),
let $\left(I_1,\ldots I_N\right)$, with $N={{r+2}\choose{r}}$, be the set of subsets of cardinality $r$ of  $\{0,\ldots, r+1\}$, lexicographically ordered. 
We consider the syzygies of degree $d-2$ of the vector of monomials
  $\left(\prod_{j\in I_1}l_j^{d-2},\ldots, \prod_{j\in I_N}l_j^{d-2}\right)$.
 We claim that any of these syzygies contains, at the entry corresponding to $ l_0^{d-2}\ldots \widehat{l_i^{d-2}}\ldots \widehat{l_j^{d-2}}\ldots l_{r+1}^{d-2} $, only the variables $x_i$ and $x_j$. In order to prove this, we set for simplicity $(i,j)=(0,1)$.
We have the identity \begin{equation}\label{eq:syzd-2}\sum_{p=1}^N s_p \prod_{j\in I_p}l_j^{d-2} = 0,\end{equation}
where $(s_p)_{p=1}^N$ is a syzygy. Note all the terms in the left hand side of \eqref {eq:syzd-2} have degree $(d-2)(r+1)$.
Consider the last summand $s_Nl_2^{d-2}\ldots l_{r+1}^{d-2}$. The only terms in the left hand side of \eqref{eq:syzd-2} containing $x_2,\ldots, x_r$ with total degree larger than $(d-2)r$
must appear in this summand, since  all the terms of the other summands contain either $x_0^{d-2} $ or $x_1^{d-2}$. 
It follows from  \eqref{eq:syzd-2}   that each term of this summand containing $x_2,\ldots, x_r$ with total degree larger than $(d-2)r$ must vanish. Hence $s_N$ is a homogeneous polynomial
of degree $d-2$ depending only on $x_0, x_1$, which proves our claim.

  Let now $g$ be such that $dH_f(g)=0$. Note that by (i) we get a syzygy of degree $d-2$ of $\left(\prod_{j\in I_1}l_j^{d-2},\ldots, \prod_{j\in I_N}l_j^{d-2}\right)$ as above. It follows that $\left(\partial_i-\partial_j\right)^2{g}$ depends only on $x_i$, $x_j$, hence all terms in $g$ containing one among $x_i^2, x_ix_j, x_j^2$ must contain only the variables $x_i$, $x_j$.
We claim that only the powers $x_i^d$ appear in $g$. In fact, assume a term $M$ with two different variables $x_i$ and $x_j$ appears in $g$. The two variables $x_i$ and $x_j$ cannot appear in $M$ both at degree one, since we have $d\geq 3$ and the remaining variables in $M$ can be only $x_i$ and $x_j$. 
 So we may suppose that $x_i^2x_j$ appears in $M$. Then take another variable $x_k$ (here we need $r\geq 2$).  Applying $\left(\partial_i-\partial_k\right)^2$ to $M$, we see that $x_j$ appears in the result, which is a contradiction because only $x_i$ and $x_k$ should appear there. 

So we get $g=\sum_{i=0}^rc_ix_i^d$ for certain scalars $c_i$, with $i=0,\ldots,r$. Then $$dH_f(g)=\sum_{i=0}^r (\sum_{j\neq i}c_j)l_0^{d-2}\ldots\widehat{l_i^{d-2}}\ldots l_r^{d-2} l_{r+1}^{d-2}.$$ When this expression vanishes it implies $\sum_{j\neq i}c_j=0$ for any $i=0,\ldots,r$, that is
\begin{equation}\label{eq:mat}\begin{pmatrix}0&1&\ldots&1\\
1&0&\ldots&1\\
\vdots&&\ddots&\vdots\\
1&1&\ldots&0\end{pmatrix}\cdot\begin{pmatrix}c_0\\c_1\\ \vdots\\c_r\end{pmatrix}=0.
\end{equation}
The matrix appearing in \eqref{eq:mat} has the eigenvalue $-1$ with multiplicity $r$. Since it is traceless the remaining eigenvalue is $r$ with multiplicity one, hence the matrix is non singular. It follows $c_i=0$ for all $i=0,\ldots,r$, hence $g=0$, so proving (ii).\end{proof}  

As an immediate consequence, we have:

\begin{corollary}\label{cor:fin}
The Hessian map $h_{d,r}$ is generically finite onto its image for $d\geq 3$, $r\geq 2$. In other words, the Hessian variety $H_{d,r}=\im(h_{d.r})$ has dimension $N(d,r)=\dim(\Sigma(d,r))$.
\end{corollary}

{
\begin{remark}\label{rem:fin} We conjecture that the Hessian map $h_{d,r}$ is birational onto its image for $d\geq 3$, $r\geq 2$, except for $h_{3,2}$,  but so far we have not been able to prove it. 

For instance we focused on the case $d=3$, $r=4$, trying to prove birationality for it. We are able to exhibit a form $f$ of degree 3 in $x_0,\ldots,x_4$ such that:\\
\begin{inparaenum}
\item [(i)] the differential of  $\hess_{3,4}$ is injective at $f$;\\
\item [(ii)] if $g$ is any cubic form in $x_0,\ldots,x_4$ such that $\hess(g)$ is proportional to $\hess(f)$ then $g$ is proportional to $f$.
\end{inparaenum}

This strongly suggests that birationality may hold, but it is not sufficient to prove it. Indeed it could be the case that for forms $h$ close to $f$ the fibre of $\hess_{3,4}$ consists of more than one point, but when $h$ tends to $f$ all the elements in the fibre, but $f$, tend to points in the indeterminacy locus. \end{remark}}

\providecommand{\bysame}{\leavevmode\hbox to3em{\hrulefill}\thinspace}
\providecommand{\MR}{\relax\ifhmode\unskip\space\fi MR }
\providecommand{\MRhref}[2]{%
  \href{http://www.ams.org/mathscinet-getitem?mr=#1}{#2}
}
\providecommand{\href}[2]{#2}


\begin{thebibliography}{10}

\bibitem{AIK}
J.~Alper, A.~V. Isaev, and N.~G. Kruzhilin, \emph{Associated forms of binary
  quartics and ternary cubics}, Transformation Groups, \textbf{21(3)} (2016),
  593--618.

\bibitem{ArtDolg}
M.~Artebani and I.~Dolgachev, \emph{The {H}esse pencil of plane cubic curves},
  Enseign. Math. (2) \textbf{55} (2009), no.~3-4, 235--273.

\bibitem{Banchi}
M.~Banchi, \emph{Rank and border rank of real ternary cubics}, Boll. Unione
  Mat. Ital. \textbf{8} (2015), no.~1, 65--80.

\bibitem{Chi}
O.~Chisini, \emph{Sui fasci di cubiche a modulo costante}, Rend. del Circolo
  Mat. di Palermo \textbf{41} (1916), 59--93.

\bibitem{Cil}
C.~Ciliberto, \emph{{{Ipersuperficie algebriche a punti parabolici e relative
  hessiane}}}, Rend. Accad. Naz. dei XL \textbf{98} (1979-80), 25--42.

\bibitem{CilRu}
C.~Ciliberto and F.~Russo, \emph{Varieties with minimal secant degree and
  linear systems of maximal dimension on surfaces}, Adv. in Math. \textbf{200}
  (2006), 1--50.

\bibitem{CilRusSim}
C.~Ciliberto, F.~Russo, and A.~Simis, \emph{Homaloidal hypersurfaces and
  hypersurfaces with vanishing {{Hessian}}}, Adv. in Math. \textbf{218} (2008),
  1759--1805.

\bibitem{deRu}
M.~De~Bondt and F.~Russo, \emph{Some results on {{Gordan-Noether}} theory},
  Pre-print (unpublished).

\bibitem{DK}
I.~Dolgachev and V.~Kanev, \emph{Polar covariants of plane cubics and
  quartics}, Adv. in Math. \textbf{98} (1993), 216--301.

\bibitem{EnrChi}
F.~Enriques and O.~Chisini, \emph{{{Lezioni sulla teoria geometrica delle
  equazioni e delle funzioni algebriche}}}, vol.~1, Zanichelli, 1915.

\bibitem{EnrChi2}
\bysame, \emph{{{Lezioni sulla teoria geometrica delle equazioni e delle
  funzioni algebriche}}}, vol.~2, Zanichelli, 1918.

\bibitem{fra}
A.~Franchetta, \emph{Sulle forme algebriche di {{$S_4$}} aventi hessiana
  indeterminata}, Rend. Mat. \textbf{13} (1954), 1--6.

\bibitem{Fulton}
W.~Fulton, \emph{Intersection theory}, Springer--Verlag, Berlin Heidelberg New
  York Tokyo, 1984.

\bibitem{FuHar}
W.~Fulton and J.~Harris, \emph{Representation theory}, Graduate Texts in
  Mathematics 129, Springer-Verlag, 1991.

\bibitem{GarRe}
A.~Garbagnati and F.~Repetto, \emph{A geometrical approach to
  {{Gordan--Noether's and Franchetta's}} contributions to a question posed by
  {{Hesse}}}, Collect. Math. \textbf{60} (2009), 27--41.

\bibitem{GoNo}
P.~Gordan and M.~N\"other, \emph{{{Ueber die algebraischen Formen, deren
  Hesse'sche Determinante identisch verschwindet}}}, Math. Ann. \textbf{10}
  (1876), 547--568.

\bibitem{M2}
D.~R. Grayson and M.~E. Stillman, \emph{Macaulay2, a software system for
  research in algebraic geometry}, Available at
  {http://www.math.uiuc.edu/Macaulay2/}.

\bibitem{Lo}
C.~Lossen, \emph{When does the {{Hessian}} determinant vanish identically? (on
  {{Gordan and Noether's}} proof of {{Hesse's}} claim)}, Bull. Braz. Math. Soc.
  \textbf{35} (2004), 71--82.

\bibitem{Ott09}
G.~Ottaviani, \emph{An invariant regarding {W}aring's problem for cubic
  polynomials}, Nagoya Math. J. \textbf{193} (2009), 95--110.

\bibitem{SemRo}
L.~Roth and J.G. Semple, \emph{{{Introduction to Algebraic Geometry}}}, Oxford
  University Press, 1949.

\bibitem{Segre}
C.~Segre, \emph{Gli ordini delle variet{\`a} che annullano i determinanti dei
  diversi gradi estratti da una data matrice}, Atti Accad. Lincei, Rend.:
  Classe di scienze fisiche, matematiche e naturali \textbf{Serie V, 9} (1900),
  253--260.

\bibitem{ShB}
N.~Shepherd-Barron, \emph{The rationality of certain spaces associated to
  trigonal curves}, Algebraic Geometry (Providence) (Amer.~Math. Soc., ed.),
  Proc. Symp. Pure Math., vol. 46, Part I, 1987.

\end{thebibliography}

\printindex

\end{document}